\documentclass[11pt,headings=small]{scrartcl}

\usepackage[usenames]{color}
\definecolor{darkred}{RGB}{139,0,0}
\definecolor{darkgreen}{RGB}{0,100,0}
\definecolor{darkmagenta}{RGB}{139,0,139}
\definecolor{darkpurple}{RGB}{110,0,180}
\definecolor{darkblue}{RGB}{40,0,200}
\definecolor{darkorange}{RGB}{255,140,0}

\usepackage[colorlinks=true,linkcolor=black,citecolor=black]{hyperref}
\usepackage[sumlimits]{amsmath}
\usepackage{amssymb,amsfonts,amsthm}
\usepackage[english]{babel}
\usepackage[latin1]{inputenc}
\usepackage[T1]{fontenc}
\usepackage{lmodern}
\usepackage{graphicx}
\usepackage{scrpage2}
\usepackage{hyperref}
\usepackage{color}
\usepackage{setspace}
\usepackage{textcomp}
\usepackage{enumerate}
\usepackage{booktabs}

\newtheorem{thm}{Theorem}[section]
\newtheorem{prp}[thm]{Proposition}
\newtheorem{cor}[thm]{Corollary}
\newtheorem{lem}[thm]{Lemma}
\newtheorem{df}[thm]{Definition}
\newtheorem*{rem}{Remark}

\def\P{\mathcal{P}}
\def\S{\mathcal{S}}
\def\R{\mathbb{R}}

\def\N{\mathbb{N}}
\def\Q{\mathcal{Q}}

\def\F{\mathbb{F}}

\def\D{\mathbb{D}}

\def\1{\mathbbm{1}}

\newcommand{\bsj}{\boldsymbol{j}}
\newcommand{\bsk}{\boldsymbol{k}}
\newcommand{\bsm}{\boldsymbol{m}}
\newcommand{\bst}{\boldsymbol{t}}
\newcommand{\bsx}{\boldsymbol{x}}
\newcommand{\bsy}{\boldsymbol{y}}
\newcommand{\bsz}{\boldsymbol{z}}
\newcommand{\bszero}{\boldsymbol{0}}
\newcommand{\bsone}{\boldsymbol{1}}
\newcommand\dint{\,{\rm d}}

\newcommand{\ld}{{\rm ld}\,}

\DeclareMathOperator{\bmo}{BMO}

\allowdisplaybreaks[1]
\clearscrheadfoot
\ihead{\headmark}
\ohead{\thepage}

\begin{document}
\pagestyle{scrheadings}
\singlespacing

\title{Discrepancy of second order digital sequences in function spaces with dominating mixed smoothness}
\author{Josef Dick\thanks{This research was supported under Australian Research Council's Discovery Projects funding scheme (project number DP150101770).}, Aicke Hinrichs, Lev Markhasin, Friedrich Pillichshammer\thanks{F.P. is supported by the Austrian Science Fund (FWF): Project F5509-N26, which is a part of the Special Research Program "Quasi-Monte Carlo Methods: Theory and Applications".}}
\date{}
\maketitle

\centerline{{\bf Gewidmet dem Andenken an Klaus Friedrich Roth 1925-2015}}
\vspace{1cm}

\begin{abstract}
The discrepancy function measures the deviation of the empirical distribution of a point set in $[0,1]^d$ from the uniform distribution. In this paper, we study the classical discrepancy function with respect to the BMO and exponential Orlicz norms, as well as Sobolev, Besov and Triebel-Lizorkin norms with dominating mixed smoothness. We give sharp bounds for the discrepancy function under such norms with respect to infinite sequences.
\end{abstract}

\centerline{\begin{minipage}[hc]{130mm}{
{\em Keywords:} $\bmo$-discrepancy, higher order digital sequences, quasi-Monte Carlo, Besov spaces, Triebel-Lizorkin spaces, Sobolev spaces, Orlicz spaces\\
{\em MSC 2010:} Primary: 11K06, 11K38; Secondary: 32A37, 42C10, 46E30, 46E35, 65C05}
\end{minipage}} 

\section{Background}
Let $d,N$ be positive integers and let $\P_{N,d}$ be a point set in the $d$-dimensional unit cube $[0,1)^d$ with $N$ elements. The discrepancy function of $\P_{N,d}$ is defined as
\begin{align}\label{deflocdisc}
D_{\P_{N,d} }(\bsx) = \frac{1}{N}\sum_{\bsz \in \P_{N,d}} \chi_{[\bszero,\bsx)}(\bsz) - x_1 \cdots x_d.
\end{align}
Here $\chi_A$ denotes the characteristic function of a subset $A\in\R^d$. Thus $D_{\P_{N,d} }(\bsx)$ denotes the difference between the actual and expected proportions
of points that fall into $[\bszero,\bsx)=[0,x_1)\times\ldots\times[0,x_d)$, where $\bsx = (x_1, \ldots, x_d) \in [0,1]^d$.

We define the $L_p$-discrepancy of $\P_{N,d}$ as the $L_p$-norm of the discrepancy function, i.e., 
\[ L_{p,N}(\P_{N,d}) = \|D_{\P_{N,d}}|L_p([0,1]^d)\| \ \ \ \ \ \mbox{ for $p \in [1,\infty]$.} \] 

Let $\S_d$ be an infinite sequence in $[0,1)^d$ and let $N \in \mathbb{N}$. We define the discrepancy function of $\S_d$ as $D_{\S_d}^N(\bsx)=D_{\P_{N,d}}(\bsx)$, where the point set $\P_{N,d}$ consists of the first $N$ elements of $\S_d$, and we define the $L_p$-discrepancy of $\S_d$ as
\[ L_{p,N}(\S_d) = \|D_{\S_d}^N|L_p([0,1]^d)\| \ \ \ \ \ \mbox{ for $p \in [1,\infty]$.} \]

The $L_p$-discrepancy is a quantitative measure for the irregularity of distribution of finite point sets and of infinite sequences. The reader is refered to \cite{BC,DT,KN} for extensive introductions to this topic. 

The conceptual difference in the concept between the discrepancy of finite point sets and infinite sequences is pointed out in \cite{Mat99} and \cite{DHMP16}: for finite point sets one is interested in the behavior of the whole set $\{\bsx_0,\bsx_1,\ldots,\bsx_{N-1}\}$ with a fixed number of elements $N$, whereas for infinite sequences one is interested in the discrepancy of all initial segments $\{\bsx_0\}$, $\{\bsx_0,\bsx_1\}$, $\{\bsx_0,\bsx_1,\bsx_2\}$, \ldots, $\{\bsx_0,\bsx_1,\bsx_2,\ldots,\bsx_{N-1}\}$, where $N=2,3,4,\ldots$. 

It is well known from results of Roth~\cite{R54} and Schmidt~\cite{S77} that for every $p \in (1,\infty]$ and every $d \in \mathbb{N}$ there exists a positive constant $c_{p,d}$ with the following property:  for every finite $N$-element point set $\P_{N,d}$ in $[0,1)^d$ with $N \ge 2$ we have 
\begin{equation}\label{lbdlpdipts}
L_{p,N}(\P_{N,d}) \ge c_{p,d} \,N^{-1} (\log N)^{\frac{d-1}{2}}.
\end{equation}
These results were extended to infinite sequences by Proinov \cite{P86} (see also \cite{DP14b} for a proof): for every $p \in(1,\infty]$ and every $d \in \mathbb{N}$, $d \ge 2$, there exists a positive constant $c_{p,d}$ with the following property: for every infinite sequence $\S_d$ in $[0,1)^d$ we have 
\begin{equation}\label{lbdlpdiseq}
L_{p,N}(\S_d) \ge c_{p,d} \,N^{-1} (\log N)^{\frac{d}{2}} \ \ \ \ \mbox{ for infinitely many $N \in \mathbb{N}$.}
\end{equation}
The lower bound \eqref{lbdlpdipts} for finite point sets is known to be best possible in the order of magnitude in $N$, i.e., for every $d,N \in \mathbb{N}$, $N \ge 2$, one can find (and explicitly construct) an $N$-element point set $\P_{N,d}$ in $[0,1)^d$ with $L_p$-discrepancy of order 
\begin{equation}\label{uplpps}
L_{p,N}(\P_{N,d}) \ll_{p,d} \,N^{-1} (\log N)^{\frac{d-1}{2}}.
\end{equation}
(For functions $f,g:D \subseteq \mathbb{N} \rightarrow \mathbb{R}$ with $g \ge 0$ the notation $f(N) \ll g(N)$ means that there exists some $C>0$ such that $f(N) \le C g(N)$ for all $N \in D$. If we want to stress that $C$ depends on some parameters, say $a,b$, then this is indicated by writing $f(N) \ll_{a,b} g(N)$. If we have $f(N) \ll g(N)$ and $g(N) \ll f(N)$, then we write $f(N) \asymp g(N)$.) 

The first proof of \eqref{uplpps} was given by Roth \cite{R80} for $p= 2$ and arbitrary $d$ and finally by Chen \cite{C80} in the general case, albeit using a probabilistic approach. Explicit constructions are due to Chen and Skriganov \cite{CS02} and Skriganov \cite{S06} (the first time that for each $1 \le p < \infty$ a construction method of an optimal point set is given), Dick and Pillichshammer \cite{DP14a}, Dick \cite{D14} (the first proof where the construction method of the point set does not depend on $p$), and Markhasin \cite{M15}.

It is also known that the lower bound \eqref{lbdlpdiseq} for infinite sequences is best possible for all $p \in (1,\infty)$. This was shown by the authors of the present paper in the recent work \cite{DHMP16} where explicit constructions $\S_d$ of infinite sequences in arbitrary dimensions $d$ were provided whose $L_p$-discrepancy satisfies
\begin{equation}\label{uplppseq}
L_{p,N}(\S_{d}) \ll_{p,d} \,N^{-1} (\log N)^{\frac{d}{2}}.
\end{equation}
The sequences and methods developed in \cite{DHMP16} are the fundamental basis for the results that will be presented in this paper (see, for example, Lemma~\ref{le131416}).

On the other hand, the $L_\infty$-discrepancy remains elusive. We have constructions of infinite sequences $\S_d$ in $[0,1)^d$ (for example order 1 digital $(t,d)$-sequences as presented in Section~\ref{sec2}, see \cite{DP10,N87,N92}) such that
\begin{eqnarray}\label{bddsts}
L_{\infty,N}(\S_d) \ll_d \,N^{-1} (\log N)^d.
\end{eqnarray}
Regarding lower bounds, it is known that there exists some $c_d >0$ and $\eta_d\in (0,\tfrac{1}{2})$ such that for every sequence $\S_d$ in $[0,1)^d$ we have \[L_{\infty,N}(\S_d) \ge c_d \,N^{-1} (\log N)^{\frac{d}{2}+\eta_d} \ \ \ \mbox{for infinitely many $N \in \mathbb{N}$.}\] This result follows from a corresponding result for finite point sets by Bilyk, Lacey and Vagharshakyan~\cite{BLV08}. For growing $d$ the exponent $\eta_d$ in this estimate tends to zero.
The quest for the exact order of the $L_{\infty}$-discrepancy in the multivariate case is a very demanding open question. Only for $d=1$ it is known that the exact order of the $L_{\infty}$-discrepancy of infinite sequences is $(\log N)/N$. This follows from \eqref{bddsts} together with a celebrated result of Schmidt~\cite{S72}.

\section{Intermediate norms}

It is now a natural and instructive question to ask what happens in intermediate spaces ``close'' to $L_{\infty}$. One standard example of such spaces is the space $\bmo^d$, which stands for {\it bounded mean oscillation}. 
For the definition of this space we need the concept of Haar functions: 

We define $\N_0=\N \cup \{0\}$ and $\N_{-1}=\N_0\cup\{-1\}$. Let $\D_j = \{0,1,\ldots, 2^j-1\}$ for $j \in \N_0$ and $\D_{-1} = \{0\}$. For $\bsj = (j_1,\dots,j_d)\in\N_{-1}^d$ let $\D_{\bsj} = \D_{j_1}\times\ldots\times \D_{j_d}$. For $\bsj\in\N_{-1}^d$ we write $|\bsj| = \max(j_1,0) + \cdots + \max(j_d,0)$.

For $j \in \N_0$ and $m \in \D_j$ we call the interval $I_{j,m} = \big[ 2^{-j} m, 2^{-j} (m+1) \big)$ the $m$-th dyadic interval in $[0,1)$ on level $j$. We put $I_{-1,0}=[0,1)$ and call it the $0$-th dyadic interval in $[0,1)$ on level $-1$. Let $I_{j,m}^+ = I_{j + 1,2m}$ and $I_{j,m}^- = I_{j + 1,2m+1}$ be the left and right half of $I_{j,m}$, respectively. For $\bsj \in \N_{-1}^d$ and $\bsm = (m_1, \ldots, m_d) \in \D_{\bsj}$ we call $I_{\bsj,\bsm} = I_{j_1,m_1} \times \ldots \times I_{j_d,m_d}$ the $\bsm$-th dyadic interval in $[0,1)^d$ on level $\bsj$. We call the number $|\bsj|$ the order of the dyadic interval $I_{\bsj,\bsm}$. Its volume is $2^{-|\bsj|}$.

Let $j \in \N_{0}$ and $m \in \D_j$. Let $h_{j,m}$ be the function on $[0,1)$ with support in $I_{j,m}$ and the constant values $1$ on $I_{j,m}^+$ and $-1$ on $I_{j,m}^-$. We put $h_{-1,0} = \chi_{I_{-1,0}}$ on $[0,1)$. The function $h_{j,m}$ is called the {\it $m$-th dyadic Haar function on level $j$}.

Let $\bsj \in \N_{-1}^d$ and $\bsm \in \D_{\bsj}$. The function $h_{\bsj,\bsm}$ given as the tensor product
\[ h_{\bsj,\bsm}(\bsx) = h_{j_1,m_1}(x_1) \cdots h_{j_d,m_d}(x_d) \]
for $\bsx = (x_1, \ldots, x_d) \in [0,1)^d$ is called a {\it dyadic Haar function} on $[0,1)^d$. The system of dyadic Haar functions $h_{\bsj,\bsm}$ for $\bsj \in \N_{-1}^d, \, \bsm \in \D_{\bsj}$ is called {\it dyadic Haar basis on $[0,1)^d$}.

It is well known that the system
\[ \left\{2^{\frac{|\bsj|}{2}}h_{\bsj,\bsm} \,:\,\bsj\in\N_{-1}^d,\,\bsm\in \D_{\bsj}\right\} \]
is an orthonormal basis of $L_2([0,1)^d)$, an unconditional basis of $L_p([0,1)^d)$ for $1 < p < \infty$ and a conditional basis of $L_1([0,1)^d)$. For any function $f\in L_2([0,1)^d)$ we have Parseval's identity
\[ \|f|L_2([0,1)^d)\|^2 = \sum_{\bsj \in \N_{-1}^d} 2^{|\bsj|} \sum_{\bsm\in \D_{\bsj}}|\langle f,h_{\bsj,\bsm}\rangle|^2, \] where $\langle \cdot , \cdot \rangle$ denotes the usual $L_2$-inner product, i.e., $\langle f ,g \rangle =\int_{[0,1]^d} f(\bsx) g(\bsx) \dint \bsx$. The terms $\langle f,h_{\bsj,\bsm}\rangle$ are called the {\it Haar coefficients} of the function $f$.\\

For an integrable function $f:[0,1]^d \rightarrow \mathbb{R}$ we define 
\begin{equation}\label{eq:defbmo}
\|f |\bmo^d\|^2 =\sup_{U \subseteq [0,1)^d} \lambda_d(U)^{-1} \sum_{\bsj \in \mathbb{N}_0^d} 2^{|\bsj|} \sum_{\bsm \in \D_{\bsj}\atop I_{\bsj,\bsm} \subseteq U}|\langle f, h_{\bsj,\bsm}\rangle|^2,
\end{equation} 
where $h_{\bsj,\bsm}$ is the $\bsm$-th dyadic Haar function on level $\bsj$, $\lambda_d$ is the $d$-dimensional Lebesgue measure and the supremum is taken over all measurable sets $U \subseteq [0,1)^d$. 

The space $\bmo^d$ contains all integrable functions $f$ with finite norm $\|f |\bmo^d\|$. Note that strictly speaking $\|f |\bmo^d\|$ is only a seminorm, since it vanishes on linear combinations of functions which are constant in one or more coordinate directions. This means that formally we need to consider a factor space over such functions. For a more detailed study of the spaces $\bmo^d$ we refer to \cite{CF80} and \cite{CWW85}.

We consider the $\bmo$-seminorm of the discrepancy function of infinite sequences and call this the $\bmo$-discrepancy. The first main result of this paper is the following.

\begin{thm}\label{thm:bmo}
Suppose that $d \in \mathbb{N}$. For every infinite sequence $\S_d$ in $[0,1)^d$, 
\begin{equation}\label{bmo:lowbd}
\|D_{\S_d}^N|\bmo^d\| \gg  \,N^{-1} (\log N)^{\frac{d}{2}}\ \ \ \mbox{ for infinitely many $N \in \mathbb{N}$.}
\end{equation}
Furthermore, there exists an infinite sequence $\S_d$ in $[0,1)^d$ such that 
\begin{equation}\label{bmo:ubd}
\|D_{\S_d}^N | \bmo^d\| \ll_d \,N^{-1} (\log N)^{\frac{d}{2}} \ \ \ \mbox{ for all $N \in \mathbb{N}\setminus\{1\}$.}
\end{equation}  
\end{thm}

The proof of this result will be presented in Section~\ref{sec_bmo2}. \\

As another example we stude the exponential Orlicz norm, which is closer to the $L_\infty$-norm than any $L_p$-norm  with $1 \le p < \infty$. Thus bounds on the exponential Orlicz norm of the discrepancy function shed more light on the behavior of the $L_\infty$-norm. For a probability space $(\Omega, P)$, let $\mathbb{E}$ denote the expectation over $(\Omega, P)$. To define the Orlicz norm, let $\psi: [0, \infty) \to [0, \infty)$ be a convex function with $\psi(x) = 0$ if and only if $x = 0$. The Orlicz norm $L^\psi$ of a $(\Omega,P)$-measurable function $f$ is now given by
\begin{equation*}
\|f | L^\psi \| = \inf\{ K > 0: \mathbb{E} \psi(|f|/K) \le 1\},
\end{equation*}
where we set $\inf \emptyset = \infty$. The Orlicz space $L^\psi$ consists of all functions $f$ with finite Orlicz norm $\|f | L^\psi \|$.

For $\alpha > 0$, one obtains the exponential Orlicz norm $\exp(L^\alpha) = L^{\psi_\alpha}$ by choosing $\psi_\alpha$ to be a convex function which equals $\exp(x^\alpha) -1$ for sufficiently large $x$, depending on $\alpha$ (for $\alpha \ge 1$ we can use $\psi_\alpha(x) = \exp(x^\alpha) - 1$ for all $x \ge 0$). More information can be found in \cite{LT77}. 

\begin{thm}\label{thm:exporlicz}
There exists an infinite sequence $\S_d$ in $[0,1)^d$ such that 
\begin{equation*}
\| D_{\S_d}^N | \exp(L^{\beta}) \| \ll_d \,N^{-1} (\log N)^{d-\frac{1}{\beta} } \ \ \ \mbox{ for all $N \in \mathbb{N}\setminus\{1\}$ and for all $\frac{2}{d-1} \le \beta < \infty$.}
\end{equation*}
\end{thm}


From Proposition~\ref{prp_exp_Lp} below in conjunction with \eqref{lbdlpdiseq} it follows that for every infinite sequence $\S_d$ in $[0,1)^d$ the discrepancy $\| D_{\S_d}^N | \exp(L^{\beta}) \|$ is at least of order $(\log N)^{d/2}/N$ for infinitely many $N \in \N$. However, a matching lower bound for  $\| D_{\S_d}^N | \exp(L^{\beta})\|$ seems to be presently beyond reach, even for finite point sets as pointed out in \cite{BM15} (see the Remark after \cite[Theorem~1.3]{BM15} for a short discussion about this problem). There is only one singular result for {\it finite} point sets $\P_{N,2}$ in dimension $d=2$, $\| D_{\P_{N,2}} | \exp(L^{\beta}) \| \ge c (\log N)^{1-1/\beta}/N$ for $2 \le \beta < \infty$ and some $c>0$, according to \cite{BLPV09}, for which the upper and lower bound match. \\

We close this section with some further technical results concerning the exponential Orlicz norm. The reader familiar with this topic may also skip straight to the beginning of the next section. 

An equivalent way of stating the exponential Orlicz norm is using the following result, which can be shown using the Taylor series expansion of $\exp(x)$ and Stirling's formula.
\begin{prp}\label{prp_exp_Lp}
For any $\beta > 0$, the following equivalence holds
\begin{equation*}
\| f | \exp(L^\beta)\| \asymp_d \sup_{p > 1} p^{-\frac{1}{\beta}} \|f | L_p([0,1]^d)\|.
\end{equation*}
\end{prp}

Our proof will make use of the Littlewood-Paley inequality, which we state in the following as a lemma.
\begin{lem}\label{lem_lp}
Let $1 < p < \infty$. Then for a function $f:[0,1)^d \to \mathbb{R}$ we have
\begin{equation*}
\| f | L_p([0,1)^d)  \| \ll_d\; p^{d/2} \| Sf | L_p([0,1)^d) \|,
\end{equation*}
where the square function $Sf$ is given by
\begin{equation*}
Sf(\bsx) = \left( \sum_{\bsj \in \mathbb{N}_{-1}^d} 2^{2|\bsj|} \sum_{\bsm \in \D_{\bsj}} |\langle f, h_{\bsj, \bsm} \rangle |^2 \chi_{I_{\bsj, \bsm}} \right)^{1/2}.
\end{equation*}
\end{lem}

An essential tool in the proof of the bound on the exponential Orlicz norm of the discrepancy function is the following hyperbolic Chang-Wilson-Wolff inequality. It relates the exponential Orlicz norm of a function to the $L_\infty$-norm of the corresponding square function. This inequality was also recently used in \cite{BM15}. We refer to this paper for a short proof.

\begin{prp}[Hyperbolic Chang-Wilson-Wolff inequality]\label{CWWIneq}
Assume that $f$ is a hyperbolic sum of multiparameter Haar functions, i.e., $f \in \mathrm{span}\{h_{\bsj,\bsm}: |\bsj| = n\}$ for some $n \in \mathbb{N}$. Then
\begin{equation*}
\| f | \exp(L^{2/(d-1)}) \| \ll \| S f | L_\infty([0,1)^d) \|.
\end{equation*}
\end{prp}

Below we use the Hyperbolic Chang-Wilson-Wolff inequality to obtain bounds on the $\exp(L^{2/(d-1)})$-norm of the discrepancy function. To obtain the general result of Theorem~\ref{thm:exporlicz} we need to interpolate between this result and the $L_\infty$-norm of the discrepancy function. Concretely, we use the following result \cite[Proposition~2.4]{BM15}, where also a short proof can be found.

\begin{prp}\label{prp_exp_int_inf}
Let $0 < \alpha < \beta < \infty$. Consider a function $f \in L_\infty([0,1)^d)$. If $f \in \exp(L^\alpha)$, then also $f \in \exp(L^\beta)$ and we have
\begin{equation*}
\| f | \exp(L^\beta) \| \ll \|f | \exp(L^\alpha) \|^{\alpha/\beta} \| f | L_{\infty}([0,1)^d) \|^{1-\alpha/\beta}.
\end{equation*}
\end{prp}

\section{Besov and Triebel-Lizorkin spaces with dominating mixed smoothness}

We are also interested in Besov spaces and Triebel-Lizorkin spaces with dominating mixed smoothness and their quasi-norm of the discrepancy function. The discrepancy of finite point sets in Besov spaces and Triebel-Lizorkin spaces and the corresponding integration orders where considered in \cite{H10,M13a,M13b,M13c,M15,T10,T10a}. 

We follow \cite{T10}. Let $\mathfrak{S}(\R^d)$ denote the Schwartz space and $\mathfrak{S}'(\R^d)$ the space of tempered distributions on $\R^d$. For $\varphi\in\mathfrak{S}(\R^d)$ we denote by $\mathcal{F}\varphi$ the Fourier transform of $\varphi$ and extend it to $\mathfrak{S}'(\R^d)$ in the usual way. For $f\in \mathfrak{S}'(\R^d)$ the Fourier transform is given as $\mathcal{F} f(\varphi) = f(\mathcal{F}\varphi),\; \varphi\in\mathfrak{S}(\R^d)$. Analogously we proceed with the inverse Fourier transform $\mathcal{F}^{-1}$.

Let $\varphi_0 \in \mathfrak{S}(\R)$ satisfy $\varphi_0(x) = 1$ for $|x| \leq 1$ and $\varphi_0(x) = 0$ for $|x| > \frac{3}{2}$. Let $\varphi_k(x) = \varphi_0(2^{-k} x) - \varphi_0(2^{-k + 1} x)$ where $x \in \R, \, k \in \N$ and $\varphi_{\bsk}(\bsx) = \varphi_{k_1}(x_1) \cdots \varphi_{k_d}(x_d)$ where $\bsk = (k_1,\ldots,k_d) \in \N_0^d$ and $\bsx = (x_1,\ldots,x_d) \in \R^d$. The set of functions $\{\varphi_{\bsk}\}$ are a dyadic resolution of unity since
\[ \sum_{{\bsk} \in \N_0^d} \varphi_{\bsk}(\bsx) = 1 \]
for all $\bsx \in \R^d$. The functions $\mathcal{F}^{-1}(\varphi_{\bsk} \mathcal{F} f)$ are entire analytic functions for every $f \in \mathfrak{S}'(\R^d)$.

Let $0 < p,q \leq \infty$ and $s \in \R$. The Besov space with dominating mixed smoothness $S_{p,q}^s B(\R^d)$ consists of all $f \in \mathfrak{S}'(\R^d)$ with finite quasi-norm
\begin{align}
\| f | S_{p,q}^s B(\R^d) \| = \left( \sum_{{\bsk} \in \N_0^d} 2^{s |\bsk| q} \| \mathcal{F}^{-1}(\varphi_{\bsk} \mathcal{F} f) | L_p(\R^d) \|^q \right)^{\frac{1}{q}}
\end{align}
with the usual modification if $q = \infty$.

Let $\mathfrak{D}([0,1)^d)$ consist of all complex-valued infinitely differentiable functions on $\R^d$ with compact support in the interior of $[0,1)^d$ and let $\mathfrak{D}'([0,1)^d)$ be its dual space of all distributions in $[0,1)^d$. The Besov space with dominating mixed smoothness $S_{p,q}^s B([0,1)^d)$ consists of all $f \in \mathfrak{D}'([0,1)^d)$ with finite quasi-norm
\begin{align}
\| f | S_{p,q}^s B([0,1)^d) \| = \inf \left\{ \| g | S_{p,q}^s B(\R^d) \| : \: g \in S_{p,q}^s B(\R^d), \: g|_{[0,1)^d} = f \right\}.
\end{align}

\begin{thm}\label{thm:besov}
Suppose that $1 \le p,q < \infty$ and $d \in \mathbb{N}$.
\begin{enumerate}
 \item[(i)] For every infinite sequence $\S_d$ in $[0,1)^d$, 
\begin{equation}\label{besov:lbd1}
\|D_{\S_d}^N | S_{p,q}^0 B([0,1)^d)\| \gg_{p,q,d} \,N^{-1} (\log N)^{\frac{d}{q}}   \ \ \ \mbox{ for infinitely many $N\in\N$.}  
\end{equation}
Furthermore, there exists an infinite sequence $\S_d$ in $[0,1)^d$ such that
\begin{equation}\label{besov:ubd1}
\|D_{\S_d}^N | S_{p,q}^0 B([0,1)^d)\| \ll_{p,q,d} \,N^{-1} (\log N)^{\frac{d}{q}}  \ \ \ \mbox{ for all $N\in\N\setminus\{1\}$.}  
\end{equation}
 \item[(ii)] Suppose further that $0 < s < 1/p$. For every infinite sequence $\S_d$ in $[0,1)^d$, 
\begin{equation}\label{besov:lbd2}
\|D_{\S_d}^N | S_{p,q}^s B([0,1)^d)\| \gg_{p,q,s,d} \,N^{s-1} (\log N)^{\frac{d-1}{q}}  \ \ \ \mbox{ for infinitely many $N\in\N$.} 
\end{equation}
Furthermore, there exists an infinite sequence $\S_d$ in $[0,1)^d$ such that
\begin{equation}\label{besov:ubd2} 
\|D_{\S_d}^N | S_{p,q}^s B([0,1)^d)\| \ll_{p,q,s,d} \,N^{s-1} (\log N)^{\frac{d-1}{q}}  \ \ \ \mbox{ for all $N\in\N\setminus\{1\}$.}  
\end{equation}
\end{enumerate}
\end{thm}
In the case $d=1$ this result has been recently shown in \cite[Theorem~1]{K15} which is based on the symmetrized van der Corput sequence.\\

Let $0 < p < \infty$, $0 < q \leq \infty$ and $s \in \R$. The Triebel-Lizorkin space with dominating mixed smoothness $S_{p,q}^s F(\R^d)$ consists of all $f \in \mathfrak{S}'(\R^d)$ with finite quasi-norm
\begin{align}
\| f | S_{p,q}^s F(\R^d) \| = \left\| \left( \sum_{{\bsk} \in \N_0^d} 2^{s |\bsk| q} | \mathcal{F}^{-1}(\varphi_{\bsk} \mathcal{F} f)(\cdot) |^q \right)^{\frac{1}{q}} \Bigg| L_p(\R^d) \right\|
\end{align}
with the usual modification if $q = \infty$.

The Triebel-Lizorkin space with dominating mixed smoothness $S_{p,q}^s F([0,1)^d)$ consists of all $f \in \mathfrak{D}'([0,1)^d)$ with finite quasi-norm
\begin{align}
\| f | S_{p,q}^s F([0,1)^d) \| = \inf \left\{ \| g | S_{p,q}^s F(\R^d) \| : \: g \in S_{p,q}^s F(\R^d), \: g|_{[0,1)^d} = f \right\}.
\end{align}

\begin{thm}\label{thm:triebel}
Suppose that $1 \le p,q < \infty$ and $d \in \mathbb{N}$.
\begin{enumerate}
 \item[(i)] For every infinite sequence $\S_d$ in $[0,1)^d$, 
\begin{equation}\label{triebel:lbd1}
\|D_{\S_d}^N | S_{p,q}^0 F([0,1)^d)\| \gg_{p,q,d} \,N^{-1} (\log N)^{\frac{d}{q}}   \ \ \ \mbox{ for infinitely many $N\in\N$.}  
\end{equation}
Furthermore, there exists an infinite sequence $\S_d$ in $[0,1)^d$ such that
\begin{equation}\label{triebel:ubd1}
\|D_{\S_d}^N | S_{p,q}^0 F([0,1)^d)\| \ll_{p,q,d} \,N^{-1} (\log N)^{\frac{d}{q}} \ \ \ \mbox{ for all $N\in\N\setminus\{1\}$.}  
\end{equation}
 \item[(ii)] Suppose further that $0<s<1/\min(p,q)$. For every infinite sequence $\S_d$ in $[0,1)^d$, 
\begin{equation}\label{triebel:lbd2}
\|D_{\S_d}^N | S_{p,q}^s F([0,1)^d)\| \gg_{p,q,s,d}  \,N^{s-1} (\log N)^{\frac{d-1}{q}} \ \ \ \mbox{ for infinitely many $N\in\N$.} 
\end{equation}
Furthermore, there exists an infinite sequence $\S_d$ in $[0,1)^d$ such that
\begin{equation}\label{triebel:ubd2} 
\|D_{\S_d}^N | S_{p,q}^s F([0,1)^d)\| \ge c_{p,q,s,d}  \,N^{s-1} (\log N)^{\frac{d-1}{q}}   \ \ \ \mbox{ for all $N\in\N\setminus\{1\}$.}  
\end{equation}
\end{enumerate}
\end{thm}

For $0<p<\infty$ the spaces $S_p^s W([0,1)^d) = S_{p,2}^s F([0,1)^d)$ are called Sobolev spaces with dominating mixed smoothness. Thus in the case $q=2$ the above theorem implies results for Sobolev norms with dominating mixed smoothness.\\

We close this section with some further technical results about Besov- and Triebel-Lizorkin spaces: The spaces $S_{p,q}^s B(\R^d), \, S_{p,q}^s F(\R^d), \, S_{p,q}^s B([0,1)^d)$ and $S_{p,q}^s F([0,1)^d)$ are quasi-Banach spaces. The following results will be helpful. The first one is an embedding result (see \cite{T10} or \cite[Corollary~1.13]{M13c}), and the second one gives a characterization of the Besov spaces in terms of Haar bases (see \cite[Theorem~2.11]{M13c}).
\begin{prp}\label{prp_emb_BF}
Let $p,q\in (0,\infty)$ and $s\in\R$. Then we have
\[ S_{\max(p,q),q}^s B([0,1)^d)\hookrightarrow S_{p,q}^s F([0,1)^d)\hookrightarrow S_{\min(p,q),q}^s B([0,1)^d). \]
\end{prp}
\begin{prp} \label{haarbesovnorm}
Let $p,q\in(0,\infty)$, $1/p-1<s<\min(1/p,1)$ and $f\in S_{p,q}^sB([0,1)^d)$. Then
\[ \|f|S_{p,q}^sB([0,1)^d)\|^q \asymp_{p,q,s,d} \sum_{\bsj\in\N_{-1}^d} 2^{|\bsj|(s-1/p+1)q}\left(\sum_{\bsm\in\D_{\bsj}} |\langle f,h_{\bsj,\bsm}\rangle|^p\right)^{q/p}. \]
\end{prp}

\section{Upper bounds}\label{sec2}

The explicit constructions in Theorems~\ref{thm:bmo}, \ref{thm:exporlicz},  \ref{thm:besov} and \ref{thm:triebel} are all the same and are based on linear algebra over the finite field $\mathbb{F}_2$. In the subsequent section we provide a detailed introduction to the infinite sequences which lead to the optimal discrepancy bounds.

\subsection{Digital nets and sequences}

Niederreiter introduced the concepts of digital nets and sequences~\cite{N87} in 1987. A detailed overview of this topic can be found in the books \cite{DP10,LP14,N92}. Here we restrict ourselves to the dyadic case. Let $\mathbb{F}_2$ be the finite field of order 2 identified with the set $\{0,1\}$ equipped with arithmetic operations modulo 2.

\paragraph{The digital construction scheme.} We begin with the definition of digital nets according to Niederreiter, which we present here in a slightly more general form. For $n,q,d \in \N$ with $q \ge n$ let $C_1,\ldots, C_d \in \mathbb{F}_2^{q \times n}$ be $q \times n$ matrices over $\mathbb{F}_2$. For $k \in \{0,\ldots ,2^n-1\}$ with dyadic expansion $k = k_0 + k_1 2 + \cdots + k_{n-1} 2^{n-1}$, where $k_j \in \{0,1\}$, we define the dyadic digit vector $\vec{k}$ as $\vec{k} = (k_0, k_1, \ldots, k_{n-1})^\top \in \mathbb{F}_2^n$ (the symbol $\top$ means the transpose of a vector or a matrix; hence $\vec{k}$ is a column-vector). Then compute
\begin{equation}\label{matrix_vec_net}
C_j \vec{k} =:(x_{j,k,1}, x_{j,k,2},\ldots,x_{j,k,q})^\top \quad \mbox{for } j = 1,\ldots, d,
\end{equation}
where the matrix vector product is evaluated over $\mathbb{F}_2$, and put
\begin{equation*}
x_{j,k} = x_{j,k,1} 2^{-1} + x_{j,k,2} 2^{-2} + \cdots + x_{j,k,q} 2^{-q} \in [0,1).
\end{equation*}
The $k$-th point $\boldsymbol{x}_k$ of the net $\P_{2^n,d}$ is given by $\boldsymbol{x}_k = (x_{1,k}, \ldots, x_{d,k})$. A net $\P_{2^n,d}$ constructed this way is called a {\it digital net (over $\mathbb{F}_2$) with generating matrices} $C_1,\ldots,C_d$. A digital net consists of $2^n$ elements in $[0,1)^d$.

Next we recall the definition of digital sequences, which are infinite versions of digital nets. Let $C_1,\ldots, C_d \in \mathbb{F}_2^{\mathbb{N} \times \mathbb{N}}$ be $\mathbb{N} \times \mathbb{N}$ matrices over $\mathbb{F}_2$. For $C_j = (c_{j,k,\ell})_{k, \ell \in \mathbb{N}}$ it is assumed that for each $\ell \in \mathbb{N}$ there exists a $K(\ell) \in \mathbb{N}$ such that $c_{j,k,\ell} = 0$ for all $k > K(\ell)$. For $k \in \N_0$ with dyadic expansion $k = k_0 + k_1 2 + \cdots + k_{m-1} 2^{m-1} \in \mathbb{N}_0$,  define the infinite dyadic digit vector of $k$ by $\vec{k} = (k_0, k_1, \ldots, k_{m-1}, 0, 0, \ldots )^\top \in \mathbb{F}_2^{\mathbb{N}}$. Then compute
\begin{equation}\label{eq_dig_seq}
C_j \vec{k}=:(x_{j,k,1}, x_{j,k,2},\ldots)^\top \quad \mbox{for } j = 1,\ldots, d,
\end{equation}
where the matrix vector product is evaluated over $\mathbb{F}_2$, and put
\begin{equation*}
x_{j,k} = x_{j,k,1} 2^{-1} + x_{j,k,2} 2^{-2} + \cdots \in [0,1).
\end{equation*}
The $k$-th point $\boldsymbol{x}_k$ of the sequence $\S_d$ is given by $\boldsymbol{x}_k = (x_{1,k}, \ldots, x_{d,k})$. We call a sequence $\S_d$ constructed this way a {\it digital sequence (over $\mathbb{F}_2$) with generating matrices} $C_1,\ldots,C_d$. Since $c_{j,k,\ell}=0$ for all $k$ large enough, the numbers $x_{j,k}$ are always dyadic rationals, i.e., have a finite dyadic expansion. 


\paragraph{Higher order nets and sequences.}


The choice of the respective generating matrices is crucial for the distribution quality of digital nets and sequences. The following definitions put some restrictions on $C_1,\ldots ,C_d$ with the aim to quantify the quality of equidistribution of the digital net or sequence.

\begin{df}\rm\label{def_net}
Let $n, q, \alpha \in \N$ with $q \ge \alpha n$ and let $t$ be an integer such that $0 \le t \le \alpha n$. Let $C_1,\ldots, C_d \in \mathbb{F}_2^{q \times n}$. Denote the $i$-th row vector of the matrix $C_j$ by $\vec{c}_{j,i} \in \mathbb{F}_2^n$. If for all $1 \le i_{j,\nu_j} < \ldots <
i_{j,1} \le q$ with \[\sum_{j = 1}^d \sum_{l=1}^{\min(\nu_j,\alpha)} i_{j,l}  \le
\alpha n - t\] the vectors
\[\vec{c}_{1,i_{1,\nu_1}}, \ldots, \vec{c}_{1,i_{1,1}}, \ldots,
\vec{c}_{d,i_{d,\nu_d}}, \ldots, \vec{c}_{d,i_{d,1}}\] are linearly independent
over $\mathbb{F}_2$, then the digital net with generating matrices
$C_1,\ldots, C_d$ is called an {\it order $\alpha$ digital $(t,n,d)$-net over $\mathbb{F}_2$}.
\end{df}


Furthermore, we consider digital sequences whose initial segments are order $\alpha$ digital $(t,n,d)$-nets over $\mathbb{F}_2$.

\begin{df}\rm\label{def_seq}
Let $\alpha \in \N$ and let $t \ge 0$ be an integer. Let $C_1,\ldots, C_d \in \mathbb{F}_2^{\mathbb{N} \times \mathbb{N}}$ and let $C_{j, \alpha n \times n}$ denote the left upper $\alpha n \times n$ submatrix of  $C_j$. If for all $n > t/\alpha$ the matrices $C_{1, \alpha n \times n},\ldots, C_{d, \alpha n \times n}$ generate an order $\alpha$ digital $(t,n,d)$-net over $\mathbb{F}_2$, then the digital sequence with generating matrices $C_1,\ldots, C_d$ is called an {\it order $\alpha$ digital $(t,d)$-sequence over $\mathbb{F}_2$}.
\end{df}


Definition~\ref{def_net} shows that if $\P_{2^n,d}$ is an order $\alpha$ digital $(t,n,d)$-net, then for any $t \le t' \le \alpha n$, $\P_{2^n,d}$ is also an order $\alpha$ digital $(t',n,d)$-net. An analogous result also applies to higher order digital sequences.

We point out that a digital net can be an order $\alpha$ digital $(t,n,d)$-net over $\mathbb{F}_2$ and at the same time an order $\alpha'$ digital $(t',n,d)$-net over $\mathbb{F}_2$ for $\alpha'\not = \alpha$. The quality parameter $t$ may depend on $\alpha$ (i.e., $t=t(\alpha)$). The same is true for digital sequences. In particular, \cite[Theorem~4.10]{D08} implies that an order $\alpha$ digital $(t,n,d)$-net is an order $\alpha'$ digital $(t',n,d)$-net for all $1 \le \alpha' \le \alpha$ with
\begin{equation}\label{eq_t_tprime}
t' = \lceil t \alpha'/\alpha \rceil \le t.
\end{equation}
The same result applies to order $\alpha$ digital $(t,d)$-sequences which are also order $\alpha'$ digital $(t',d)$-sequences for $1 \le \alpha' \le \alpha$ and $t'$ as above. In other words, $t(\alpha') = \lceil t(\alpha) \alpha'/\alpha \rceil$ for all $1 \le \alpha' \le \alpha$. For more details consult \cite{D07,D08} or \cite[Chapter~15]{DP10}.

Higher order digital nets and sequences have also a geometrical interpretation. Roughly speaking the definitions imply that special intervals or unions of intervals of prescribed volume contain the correct share of points with respect to a perfect uniform distribution. See \cite{N92,DP10} for the classical case $\alpha=1$ and \cite{DB09}  
for general $\alpha$. See also Lemma~\ref{fairint} below.

\paragraph{Explicit constructions of order 2 digital sequences}

Explicit constructions of order $\alpha$ digital nets and sequences have been provided by Dick~\cite{D07,D08}. For our purposes it suffices to consider only $\alpha=2$. 

We start with generating matrices $C_1, \ldots, C_{2 d}$ of a digital net or sequence. Let $\vec{c}_{j,k}$ denote the $k$-th row of $C_j$. We now define matrices $E_1,\ldots, E_d$, where the $k$-th row of $E_j$ is given by $\vec{e}_{j,k}$, in the following way. For all $j \in \{1,2,\ldots,d\}$, $u \in \N_0$ and $v \in \{1,2\}$ let
\begin{equation*}
\vec{e}_{j,2 u + v} = \vec{c}_{2 (j-1) + v, u+1}.
\end{equation*}
We illustrate the construction for $d=1$. In this case we have
\[C_1=\left(\begin{array}{c} 
             \vec{c}_{1,1}\\
             \vec{c}_{1,2}\\
             \vdots 
            \end{array}\right), \ C_2=\left(\begin{array}{c} 
             \vec{c}_{2,1}\\
             \vec{c}_{2,2}\\
             \vdots 
            \end{array}\right) \ \Rightarrow \ 
           E_1=\left(\begin{array}{c} 
             \vec{c}_{1,1}\\
             \vec{c}_{2,1}\\
             \vec{c}_{1,2}\\
             \vec{c}_{2,2}\\
             \vdots 
            \end{array}\right).\]
This procedure is called \emph{interlacing} (in this case the so-called interlacing factor is $2$).

Recall that above we assumed that $c_{j,k,\ell}=0$ for all $k > K(\ell)$. Let $E_j = (e_{j,k,\ell})_{k, \ell \in \mathbb{N}}$. Then the construction yields that $e_{j,k,\ell} = 0$ for all $k > 2 K(\ell)$. 

From \cite[Theorem~4.11 and Theorem~4.12]{D07} we obtain the following result.
\begin{prp}
If $C_1,\ldots,C_{2d} \in\mathbb{F}_2^{\N \times \N}$ generate an order 1 digital $(t',2d)$-sequence over $\mathbb{F}_2$, then $E_1,\ldots,E_d \in\mathbb{F}_2^{\N \times \N}$ generate an order 2 digital $(t,d)$-sequence over $\mathbb{F}_2$ with \[t = 2 t' + d.\]
\end{prp}

Examples of explicit constructions of suitable generating matrices over $\mathbb{F}_2$ were obtained by Sobol'~\cite{S67}, Niederreiter~\cite{N87,N92}, Niederreiter-Xing~\cite{NX96} and others. An overview is presented in \cite[Chapter~8]{DP10}. Any of these constructions is sufficient for our purpose, however, for completeness, we briefly describe a special case of Tezuka's construction~\cite{T93}, which is a generalization of Sobol's construction~\cite{S67} and Niederreiter's construction~\cite{N87} of the generating matrices. 

We explain how to construct the entries $c_{j,k,\ell} \in \mathbb{F}_2$ of the generator matrices $C_j = (c_{j,k,\ell})_{k,\ell \ge 1}$ for $j=1,2,\ldots,d'$ (for our purpose $d'=2d$). To this end choose the polynomials $p_j \in \mathbb{F}_2[x]$ for $j =1,\ldots,d'$ to be the $j$-th irreducible polynomial in a list of irreducible polynomials over $\mathbb{F}_2$ that is sorted in increasing order according to their degree $e_j = \deg(p_j)$, that is, $e_1 \le e_2 \le \ldots \le e_{d'}$ (the ordering of polynomials with the same degree is irrelevant). 

Let $j \in \{1,\ldots,d'\}$ and $k \in \mathbb{N}$. Take $i-1$ and $z$ to be respectively the main term and remainder when we divide $k-1$ by $e_j$, so that   $k-1  = (i-1) e_j + z$, with $0 \le z < e_j$. Now consider the Laurent series expansion
\begin{equation*}
\frac{x^{e_j-z-1}}{p_j(x)^i} = \sum_{\ell =1}^\infty a_\ell(i,j,z) x^{-\ell} \in \mathbb{F}_2((x^{-1})).
\end{equation*}
For $\ell \in \mathbb{N}$ we set
\begin{equation}\label{def_sob_mat}
c_{j,k,\ell} = a_\ell(i,j,z).
\end{equation}
Every digital sequence with generating matrices $C_j = (c_{j,k,\ell})_{k,\ell \ge 1}$ for $j=1,2,\ldots,d'$ found in this way is a special instance of a Sobol' sequence which in turn is a special instance of so-called generalized Niederreiter sequences (see \cite[(3)]{T93}). Note that in the construction above we always have $c_{j,k,\ell}=0$ for all $k > \ell$. The $t$-value for these sequences is known to be $t = \sum_{j=1}^{d'} (e_j-1)$, see \cite[Chapter~4.5]{N92} for the case of Niederreiter sequences.

\begin{rem}
Let $C_1,\ldots,C_{2d}$ be $\mathbb{N}\times \mathbb{N}$ matrices which are constructed according to Tezuka's method as described above. Let $E_1,\ldots,E_d$ be the generator matrices of the corresponding order 2 digital sequence. Then we always have $e_{j,k,\ell}=0$ for all $k > 2 \ell$, where $e_{j,k,\ell}$ is the entry in row $k$ and column $\ell$ of the matrix $E_j$.  
\end{rem}

We point out that the explicit construction of the order $2$ digital $(t,d)$-sequences $\S_d$ shown above, with generating matrices $E_i=(e_{i,k,\ell})_{k,\ell \ge 1}$ for which $e_{i,k,\ell}=0$ for all $k > 2\ell$ and for all $i\in \{1,2,\ldots,d\}$, does not depend on the parameters $p, q, s$ in Theorems~\ref{thm:exporlicz}, \ref{thm:besov} and \ref{thm:triebel}. Our explicit construction, which is based on Tezuka's construction and the interlacing of the generating matrices, is also extensible in the dimension, i.e., if we have constructed the sequence $\S_d$, we can add one more coordinate to obtain the sequence $\S_{d+1}$. This means that we can define a sequence $\S_{\infty}$ of points in $[0,1)^{\mathbb{N}}$ and obtain the sequence $\S_d$, $d \in \mathbb{N}$, by projecting $\S_{\infty}$ to the first $d$ coordinates (cf. \cite{D14}).

\subsection{Some auxiliary results} 

A connection between higher order digital nets over $\mathbb{F}_2$ and dyadic intervals is given by the following result (see \cite[Lemma 3.1]{DHMP16}).

\begin{lem}\label{fairint}
Let $\P_{2^n,d}$ be an order $\alpha$ digital $(t,n,d)$-net over $\mathbb{F}_2$. Then every dyadic interval of order $n-\lceil t/\alpha\rceil$ contains at most $2^{\lceil t/\alpha\rceil}$ points of $\P_{2^n,d}$.   
\end{lem}

\begin{proof}
As already mentioned, every order $\alpha$ digital $(t,n,d)$-net over $\mathbb{F}_2$ is an order~1 digital $(\lceil t/\alpha \rceil,n,d)$-net over $\mathbb{F}_2$. Then every dyadic interval of order $n-\lceil t/\alpha\rceil$ contains exactly $2^{\lceil t/\alpha\rceil}$ points of $\P_{2^n,d}$ (see \cite{DP10,N92}). 
\end{proof}

The following lemma, implicitly shown in \cite{DHMP16}, is utmost important for our analysis. Throughout the paper let $\ld$ denote the logarithm in base $2$, i.e. $\ld x=(\log x)/\log 2$.

\begin{lem}\label{le131416}
Let $\S_d$ be an order 2 digital $(t,d)$-sequence over $\mathbb{F}_2$ with generating matrices $E_i=(e_{i,k,\ell})_{k,\ell \ge 1}$ for which $e_{i,k,\ell}=0$ for all $k > 2\ell$ and for all $i\in \{1,2,\ldots,d\}$.
\begin{enumerate}[(i)]
\item If $|\bsj|+t/2\geq \ld N$ we have: if $I_{\bsj,\bsm}$ contains points of $\P_{N,d}$, then \[|\langle D_{\S_d}^N,h_{\bsj,\bsm}\rangle| \ll  2^{t/2}\,N^{-1} 2^{-|\bsj|}.\] At least $2^{|\bsj|}-N$ such intervals contain no points of $\P_{N,d}$ and in such cases we have 
\[|\langle D_{\S_d}^N,h_{\bsj,\bsm}\rangle| \ll 2^{-2|\bsj|}.\]
\item If $|\bsj|+t/2< \ld N$, more precisely $n_\mu \le |\bsj|+t/2< n_{\mu+1}$, where $N=2^{n_r}+\cdots+2^{n_1}$ with $n_r > \ldots >n_1 \ge 0$, then \[|\langle D_{\S_d}^N,h_{\bsj,\bsm}\rangle| \ll 2^t\,N^{-1} \left(2^{-|\bsj|} +  (2n_{\mu+1} - t - 2|\bsj|)^{d-1}2^{-n_{\mu+1}}\right).\]
\end{enumerate}
\end{lem}

\begin{proof}
This is \cite[Eq. (13), (14) and (16)]{DHMP16}. 
\end{proof}

\begin{lem}\label{index_dim_red}
Let $r\in\N_0$ and $s\in\N$. Then
\[ \#\{(a_1,\ldots,a_s)\in\N_0^s:\; a_1 + \cdots + a_s = r\} \leq (r + 1)^{s-1}. \]
\end{lem}
For a proof of this result we refer to \cite[Proof of Lemma~16.26]{DP10}.

\begin{lem}\label{index_dim_red_log}
 Let $K\in\N$, $A>1$ and $q,s \ge 0$. Then we have
\[ \sum_{r = 0}^{K-1} A^r (K-r)^q r^s \ll A^K\,K^s, \]
 where the implicit constant is independent of $K$.
\end{lem}
For a proof we refer to \cite[Lemma~5.2]{M15}.

The subsequent two lemmas are required in order to estimate the Haar coefficients of the discrepancy function. The first one is a special case of \cite[Lemma~5.1]{M13b}.

\begin{lem}\label{lem_haar_coeff_vol}
Let $f(\bsx) = x_1\cdots x_d$ for $\bsx=(x_1,\ldots,x_d)\in [0,1)^d$. Let $\bsj\in\N_{-1}^d$ and $\bsm\in\D_{\bsj}$. Then $|\langle f,h_{\bsj,\bsm}\rangle|\asymp 2^{-2|\bsj|}$.
\end{lem}

The next lemma is a special case of \cite[Lemma~5.2]{M13b}.
\begin{lem} \label{lem_haar_coeff_number}
Let $\bsz = (z_1,\ldots,z_d) \in [0,1)^d$ and $g(\bsx) = \chi_{[\bszero,\bsx)}(\bsz)$ for $\bsx = (x_1, \ldots, x_d) \in [0,1)^d$. Let $\bsj\in\N_{-1}^d$ and $\bsm\in\D_{\bsj}$. Then $\langle g,h_{\bsj,\bsm}\rangle = 0$ if $\bsz$ is not contained in the interior of the dyadic interval $I_{\bsj,\bsm}$. If $\bsz$ is contained in the interior of $I_{\bsj,\bsm}$ then $|\langle g,h_{\bsj,\bsm}\rangle|\ll 2^{-|\bsj|}$.
\end{lem}

\subsection{The proof of the upper bound in Theorem~\ref{thm:bmo}}\label{sec_bmo1}

\begin{proof}[Proof of Eq. \eqref{bmo:ubd} in Theorem~\ref{thm:bmo}]
The proof uses ideas from \cite{BM15}. Let $N=2^{n_r}+\cdots+2^{n_1}$ with $n_r>\ldots>n_1\geq 0$ and let $\P_{N,d}$ be the point set which consists of the first $N$ terms of $\S_d$. We need to prove that \[\lambda_d(U)^{-1} \sum_{\bsj \in \mathbb{N}_0^d} 2^{|\bsj|} \sum_{\bsm \in \D_{\bsj}\atop I_{\bsj,\bsm} \subseteq U}|\langle D_{\S_d}^N, h_{\bsj,\bsm}\rangle|^2 \ll_d 2^{2 t} \,N^{-2} (\log N)^{d}.\]

It is clear that for each $U \subseteq [0,1)^d$ and $\bsj \in \mathbb{N}_0^d$ there are at most $2^{|\bsj|} \lambda_d(U)$ values of $\bsm \in \D_{\bsj}$ such that $I_{\bsj,\bsm} \subseteq U$.

We split the above sum over $\bsj \in \mathbb{N}_0^d$ into three parts according to the size of $|\bsj|$.

First we consider those $\bsj$ for which $\ld N -t/2 \le |\bsj| < \ld N$. From Lemma~\ref{le131416} we know that \[|\langle D_{\S_d}^N,h_{\bsj,\bsm}\rangle| \ll 2^{t/2}\,N^{-1} 2^{-|\bsj|}\] if $I_{\bsj,\bsm}$ contains points of $\P_{N,d}$ and \[|\langle D_{\S_d}^N,h_{\bsj,\bsm}\rangle| \ll 2^{-2|\bsj|}\] otherwise. With these estimates we obtain
\begin{eqnarray*}
\lefteqn{\lambda_d(U)^{-1} \sum_{\bsj \in \mathbb{N}_0^d \atop \ld N -t/2 \le |\bsj| < \ld N} 2^{|\bsj|} \sum_{\bsm \in \D_{\bsj}\atop I_{\bsj,\bsm} \subseteq U}|\langle D_{\S_d}^N, h_{\bsj,\bsm}\rangle|^2}\\
& \ll & \lambda_d(U)^{-1} \sum_{\bsj \in \mathbb{N}_0^d \atop \ld N -t/2 \le |\bsj| < \ld N} 2^{|\bsj|} \left(\sum_{\bsm \in \D_{\bsj}\atop {I_{\bsj,\bsm} \subseteq U \atop I_{\bsj,\bsm}\cap \P_{N,d}\not=\emptyset}} 2^t\,N^{-2} 2^{-2 |\bsj|} +   \sum_{\bsm \in \D_{\bsj}\atop {I_{\bsj,\bsm} \subseteq U \atop I_{\bsj,\bsm}\cap \P_{N,d}=\emptyset}} 2^{-4 |\bsj|}\right)\\
& \le & \lambda_d(U)^{-1} \sum_{\bsj \in \mathbb{N}_0^d \atop \ld N -t/2 \le |\bsj| < \ld N} 2^{|\bsj|} \max\left(2^t\,N^{-2} 2^{-2 |\bsj|} , 2^{-4 |\bsj|}\right) 2^{|\bsj|}\lambda_d(U)\\
& \ll & 2^t \,N^{-2} \sum_{\bsj \in \mathbb{N}_0^d \atop \ld N -t/2 \le |\bsj| < \ld N} 1\\
& \ll & 2^t \,N^{-2} \sum_{k < \ld N} k^{d-1}\\
& \ll & 2^t \,N^{-2} (\log N)^d.
\end{eqnarray*}

Now we consider those $\bsj \in \mathbb{N}_0^d$ for which $|\bsj| < \ld N -t/2$. More precisely assume that we have
\[ n_{\mu} \le |\bsj|+t/2 < n_{\mu+1} \]
for some $\mu\in \{0, 1,\ldots,r\}$, where we set $n_0 = 0$ and $n_{r+1} = \ld N$. From Lemma~\ref{le131416} we know that 
\[ |\langle D_{\S_d}^N,h_{\bsj,\bsm}\rangle| \ll 2^t\,N^{-1} \left(2^{-|\bsj|} +  (2n_{\mu+1} - t - 2|\bsj|)^{d-1}2^{-n_{\mu+1}}\right).\]
Then we have
\begin{eqnarray*}
\lefteqn{\lambda_d(U)^{-1} \sum_{\bsj \in \mathbb{N}_{0}^d \atop  |\bsj| < \ld N -t/2} 2^{|\bsj|} \sum_{\bsm \in \D_{\bsj}\atop I_{\bsj,\bsm} \subseteq U}|\langle D_{\S_d}^N, h_{\bsj,\bsm}\rangle|^2}\\
& \ll & \sum_{\mu=0}^r \sum_{\bsj \in \mathbb{N}_{0}^d \atop n_\mu \le |\bsj| + t/2 < n_{\mu+1}}   2^{2|\bsj|} \,2^{2t}\,N^{-2} \left(2^{-|\bsj|} +  (2n_{\mu+1} - t - 2|\bsj|)^{d-1}2^{-n_{\mu+1}}\right)^2.
\end{eqnarray*}
From here we can proceed in the same way as in \cite[Proof of Theorem~2.2]{DHMP16}. This way we obtain 
\[\lambda_d(U)^{-1} \sum_{\bsj \in \mathbb{N}_0^d \atop  |\bsj| < \ld N -t/2} 2^{|\bsj|} \sum_{\bsm \in \D_{\bsj}\atop I_{\bsj,\bsm} \subseteq U}|\langle D_{\S_d}^N, h_{\bsj,\bsm}\rangle|^2 \ll 2^{2t} \,N^{-2} (\log N)^d.\]

Finally we need to study the case $|\bsj| \ge \ld N$ which is the most involved one among the three cases considered. As in \cite{BM15} we treat the counting part and the volume part of the discrepancy function separately. For the volume part $L(\bsx)=x_1\cdots x_d$ it is easy to show that $|\langle L , h_{\bsj,\bsm}\rangle| \asymp 2^{-2|\bsj|}$. Thus we have
\begin{eqnarray*}
\lambda_d(U)^{-1} \sum_{\bsj \in \mathbb{N}_0^d \atop |\bsj| \ge \ld N} 2^{|\bsj|} \sum_{\bsm \in \D_{\bsj}\atop I_{\bsj,\bsm}\subseteq U} |\langle L , h_{\bsj,\bsm}\rangle|^2 \ll  \sum_{\bsj \in \mathbb{N}_0^d \atop |\bsj| \ge \ld N} 2^{-2|\bsj|}\ll \sum_{k \ge \ld N} k^{d-1}\,2^{-2k} \ll N^{-2} (\log N)^{d-1}.
\end{eqnarray*}

Let $C(\bsx)=N^{-1}\sum_{\bsz \in \P_{N,d}} \chi_{[\bszero,\bsx)}(\bsz)$ be the counting part of the discrepancy function.

Let $\mathcal{J}$ be the family of all dyadic intervals $I_{\bsj,\bsm} \subseteq U$ with $|\bsj| \ge \ld N$ such that $\langle C, h_{\bsj,\bsm}\rangle \not=0$. Consider the subfamily $\widetilde{\mathcal{J}} \subseteq \mathcal{J}$ of all maximal (with respect to inclusion) dyadic intervals in $\mathcal{J}$. We show that we have 
\begin{equation}\label{toshow}
\sum_{I_{\bsj,\bsm}\in \widetilde{\mathcal{J}}} \lambda_d(I_{\bsj,\bsm}) \ll (\log N)^{d-1} \lambda_d(U).
\end{equation}

To see this consider an interval $I_{\bsj,\bsm} \in \mathcal{J}$. The condition $\langle C,h_{\bsj,\bsm}\rangle \not=0$ implies that the interior of $I_{\bsj,\bsm}$ contains at least one point of $\P_{N,d}$. Now we need the assumption that the generating matrices $E_i=(e_{i,k,\ell})_{k,\ell \ge 1}$ are chosen such that $e_{i,k,\ell}=0$ for all $k > 2 \ell$ and all $i\in \{1,2,\ldots,d\}$. This implies that the points of $\P_{N,d}$ have binary coordinates of length $2n$ (where $n=\lfloor \ld N \rfloor +1$). Since the interior of $I_{\bsj,\bsm}$ contains at least one point of $\P_{N,d}$ we obtain that each side of $I_{\bsj,\bsm}$ has length at least $2^{-2n}$, i.e., $j_i \in \{0,1,\ldots,2n\}$ for all $i\in \{1,2,\ldots,d\}$.

Now fix $r_1,r_2,\ldots,r_{d-1} \in \{0,1,\ldots,2n\}$ and consider the family $\widetilde{\mathcal{J}}_{r_1,\ldots,r_{d-1}} \subseteq \widetilde{\mathcal{J}}$  of those intervals $I_{\bsj,\bsm}\in \mathcal{J}$ for which $j_i=r_i$ for $i \in \{1,2,\ldots,d-1\}$ (i.e., with fixed lengths of their first $d-1$ sides). If two of these intervals have a nonempty intersection, then their first $d-1$ sides have to coincide and hence one would have to be contained in the other. This however contradicts the maximality of the intervals in $\widetilde{\mathcal{J}}$. Thus we have shown that the intervals in  $\widetilde{\mathcal{J}}_{r_1,\ldots,r_{d-1}}$ are pairwise disjoint. From this observation we find that \[\sum_{I_{\bsj,\bsm}\in \widetilde{\mathcal{J}}} \lambda_d(I_{\bsj,\bsm}) = \sum_{r_1,\ldots,r_{d-1}=0}^{2n} \sum_{I_{\bsj,\bsm}\in \widetilde{\mathcal{J}}_{r_1,\ldots,r_{d-1}}} \lambda_d(I_{\bsj,\bsm})  \le \sum_{r_1,\ldots,r_{d-1}=0}^{2n} \lambda_d(U) \ll (2n)^{d-1} \lambda_d(U),\] which implies \eqref{toshow}.

For dyadic intervals $J$, we define \[C^J(\bsx)=\sum_{\bsz \in \P_{N,d} \cap J} \chi_{[\bszero,\bsx)}(\bsz).\] It is easy to see that $\langle C, h_{\bsj,\bsm}\rangle = \langle C^J, h_{\bsj,\bsm}\rangle$ whenever $I_{\bsj,\bsm}\subseteq J$. From $n=\lfloor \ld N \rfloor +1$ we know that $N < 2^n$. Then the point set $\P_{N,d}$ of the first $N$ elements of $\S_d$ is a subset of $\P_{2^n,d}$ consisting of the first $2^n$ elements of $\S_d$. From the construction of $\S_d$ it follows that $\P_{2^n,d}$ is an order $2$ digital $(t,n,d)$-net over $\mathbb{F}_2$. Therefore, according to Lemma~\ref{fairint}, in an interval $I_{\bsj,\bsm}$ there are at most $2^{\lceil t/2\rceil}$ points of $\P_{2^n,d}$ and hence  there are at most $2^{\lceil t/2\rceil}$ points of $\P_{N,d}$ in $I_{\bsj,\bsm}$. Therefore, for any $J \in \widetilde{\mathcal{J}}$ we have $0 \le C^J(\bsx) \le 2^{\lceil t/2 \rceil} N^{-1}$, which implies
\begin{equation}\label{hjkl}
\|C^J\|_{L_2(J)} \le  \|2^{\lceil t/2 \rceil} N^{-1} \|_{L_2(J)} \le N^{-1} 2^{\lceil t/2 \rceil} \lambda_d(J)^{1/2}.
\end{equation}

Using the orthogonality of Haar functions, Bessel's inequality, \eqref{toshow}, and \eqref{hjkl}, we obtain
\begin{eqnarray*}
\lambda_d(U)^{-1}\sum_{I_{\bsj,\bsm} \in \mathcal{J}}2^{|\bsj|}|\langle C, h_{\bsj,\bsm}\rangle|^2 & \le &  \lambda_d(U)^{-1}\sum_{J \in \widetilde{\mathcal{J}}} \sum_{I_{\bsj,\bsm} \subseteq J} 2^{|\bsj|}|\langle C^J, h_{\bsj,\bsm}\rangle|^2 \\
& \le & \lambda_d(U)^{-1}\sum_{J \in \widetilde{\mathcal{J}}} \| C^J\|^2_{L_2(J)}\\
& \le & \lambda_d(U)^{-1}\, 2^{t+1} \,N^{-2} \sum_{J \in \widetilde{\mathcal{J}}} \lambda_d(J) \\
& \ll & 2^{t+1} \,N^{-2} (\log N)^{d-1}.
\end{eqnarray*}
The result now follows by putting the three parts together and taking square roots.
\end{proof}

\subsection{The proof of Theorem~\ref{thm:exporlicz}}\label{sec_exp1}

\begin{proof}[Proof of Theorem~\ref{thm:exporlicz}]
We prove now the bound on the exponential Orlicz norm of the discrepancy function. Our proof uses ideas from \cite{BM15}. Let again $N = 2^{n_r} + \cdots + 2^{n_1}$ with $n_r > \ldots > n_1 \ge 0$ and let $\P_{N,d}$ be the point set which consists of the first $N$ terms of $\S_d$. We first restrict ourselves to the $\exp(L^{2/(d-1)})$-norm. The general result will follow directly by interpolating between the $\exp(L^{2/(d-1)})$-norm and the $L_\infty$-norm using Proposition~\ref{prp_exp_int_inf} and the fact that for digital sequences we have
\begin{equation*}
\|D^N_{\S_d} | L_\infty([0,1)^d) \| \ll_d 2^t \,N^{-1} (\log N)^d,
\end{equation*}
which follows from the fact that $\S_d$ is a digital $(t',d)$-sequence (i.e., order $1$ digital sequence) and \cite[Theorem~4.13]{N92}.

We first use the Haar series expansion for the discrepancy function, the triangle inequality and the Chang-Wilson-Wolff inequality (Lemma~\ref{CWWIneq}) to obtain
\begin{align}
\| D^N_{\S_d} | \exp(L^{2/(d-1)}) \| &= \left\| \sum_{k=0}^\infty \sum_{\substack{ \bsj \in \mathbb{N}_{-1}^d \\ |\bsj| = k } } 2^{|\bsj|} \sum_{\bsm \in  \D_{\bsj}} \langle D^N_{\S_d}, h_{\bsj, \bsm} \rangle h_{\bsj,\bsm}  \Bigg| \exp(L^{2/(d-1)}) \right\| \nonumber \\ &\le \sum_{k = 0}^\infty \left\| \sum_{\substack{ \bsj \in \mathbb{N}_{-1}^d \\ |\bsj| = k}} 2^{|\bsj|} \sum_{\bsm \in  \D_{\bsj}} \langle D^N_{\S_d}, h_{\bsj, \bsm} \rangle h_{\bsj,\bsm}  \Bigg| \exp(L^{2/(d-1)}) \right\| \nonumber \\ &\le\sum_{k=0}^\infty  \left\|  \left( \sum_{\substack{ \bsj \in \mathbb{N}_{-1}^d \\ |\bsj | = k }} 2^{2|\bsj|}  \sum_{\bsm \in \D_{\bsj}} | \langle D^N_{\S_d}, h_{\bsj, \bsm} \rangle |^2 \chi_{I_{\bsj, \bsm}} \right)^{1/2}  \Bigg| L_\infty([0,1)^d) \right\|. \label{exp_bound_inf}
\end{align}

To prove a bound on \eqref{exp_bound_inf}, we divide the sum over $k$ into three parts, i.e. we consider the $L_\infty$-norm of the square function where $|\bsj|$ lies in a certain range.

For $\ld N - t/2 \le |\bsj| < \ld N$, we have from Lemma~\ref{le131416} that
\begin{equation*}
|\langle D^N_{\S_d}, h_{\bsj, \bsm} \rangle | \ll 2^{t/2}\,N^{-1}2^{-|\bsj|}
\end{equation*}
if $I_{\bsj, \bsm}$ contains points of $\P_{N,d}$ and
\begin{equation*}
|\langle D^N_{\S_d}, h_{\bsj, \bsm}  \rangle | \ll 2^{-2|\bsj|}
\end{equation*}
otherwise. Hence, for $\ld N - t/2 \le k < \ld N$ we have
\begin{align*}
\sup_{\bsx \in [0,1)^d} \left( \sum_{\substack{ \bsj \in \mathbb{N}_{-1}^d \\ |\bsj | = k }} 2^{2|\bsj|}  \sum_{\bsm \in \D_{\bsj}} | \langle D^N_{\S_d}, h_{\bsj, \bsm} \rangle |^2 \chi_{I_{\bsj, \bsm}} \right)^{1/2}  \ll & \left( \sum_{\substack{ \bsj \in \mathbb{N}_{-1}^d \\ |\bsj | = k }}  \max\left\{ 2^t\,N^{-2} ,  2^{-2 |\bsj|} \right\} \right)^{1/2} \\ \le & \left( 2^{t}\,N^{-2} \sum_{\substack{ \bsj \in \mathbb{N}_{-1}^d \\ |\bsj | = k }} 1 \right)^{1/2} \\ \ll & 2^{t/2} \,N^{-1} (\log N)^{(d-1)/2}.
\end{align*}
Therefore
\begin{align*}
\sum_{\substack{k \in \mathbb{N}_0 \\ \ld N - t/2 \le k < \ld N}}  \left\|  \left( \sum_{\substack{ \bsj \in \mathbb{N}_{-1}^d \\ |\bsj | = k }} 2^{2|\bsj|}  \sum_{\bsm \in \D_{\bsj}} | \langle D^N_{\S_d}, h_{\bsj, \bsm} \rangle |^2 \chi_{I_{\bsj, \bsm}} \right)^{1/2}  \Bigg| L_\infty([0,1)^d) \right\| \\
 \ll  2^t \,N^{-1} (\log N)^{(d-1)/2}.
\end{align*}

Next we consider the case where $0 \le |\bsj| < \ld N - t/2$. Assume that we have
\begin{equation*}
n_{\mu} \le |\bsj| +t/2 < n_{\mu+1},
\end{equation*}
for some $\mu \in \{0, 1, \ldots, r\}$, where we set $n_{r+1} = \ld N$ and $n_0 = 0$. From Lemma~\ref{le131416} we have that
\begin{equation*}
|\langle D^N_{\S_d}, h_{\bsj, \bsm} \rangle | \ll 2^t\,N^{-1} \left(2^{-|\bsj|} + (2 n_{\mu+1} -t-2|\bsj| )^{d-1}2^{-n_{\mu+1}}\right).
\end{equation*}
Hence, for $n_{\mu} \le k + t/2 < n_{\mu+1}$, we have
\begin{align*}
& \sup_{\bsx \in [0,1)^d} \left(\sum_{\substack{\bsj \in \mathbb{N}_{-1}^d \\ |\bsj| = k}} 2^{2|\bsj|} \sum_{\bsm \in \D_{\bsj}} | \langle D^N_{\S_d}, h_{\bsj, \bsm} \rangle |^2 \chi_{I_{\bsj, \bsm}} \right)^{1/2} \\ &\ll 2^t\,N^{-1} \left(\sum_{\substack{\bsj \in \mathbb{N}_{-1}^d \\ |\bsj| = k}} 2^{2|\bsj|} \left(2^{-|\bsj|} + (2 n_{\mu+1} -t-2|\bsj| )^{d-1}2^{-n_{\mu+1}}\right)^2 \right)^{1/2}, 
\end{align*}
and therefore we need to estimate
\begin{align*}
& \sum_{\substack{k \in \mathbb{N}_0 \\ k \le \ld N - t/2 }} \left\|  \left( \sum_{\substack{ \bsj \in \mathbb{N}_{-1}^d \\ |\bsj | = k }} 2^{2|\bsj|}  \sum_{\bsm \in \D_{\bsj}} | \langle D^N_{\S_d}, h_{\bsj, \bsm} \rangle |^2 \chi_{I_{\bsj, \bsm}} \right)^{1/2}  \Bigg| L_\infty([0,1)^d) \right\| \\ \ll & 2^t\,N^{-1} \sum_{\mu=0}^r \sum_{\substack{k \in \mathbb{N}_0 \\ n_{\mu} \le k + t/2 < n_{\mu+1}}} \left( \sum_{\substack{\bsj \in \mathbb{N}_{-1}^d \\ |\bsj| = k}} 2^{2|\bsj|} \left(2^{-|\bsj|} + (2 n_{\mu+1} -t-2|\bsj| )^{d-1}2^{-n_{\mu+1}}\right)^2 \right)^{1/2} \\ \ll & 2^t\,N^{-1} \sum_{\mu=0}^r \sum_{\substack{k \in \mathbb{N}_0 \\ n_{\mu} \le k + t/2 < n_{\mu + 1}}} \left( k^{d-1}  2^{2k } \left(2^{-k} + (2 n_{\mu+1} -t-2 k )^{d-1}2^{-n_{\mu+1}}\right)^2 \right)^{1/2} \\ = &  2^t\,N^{-1} \sum_{\mu=0}^r \sum_{\substack{k \in \mathbb{N}_0 \\ n_{\mu} \le k + t/2 < n_{\mu+1}}} k^{(d-1)/2}  \left(1 + (2 n_{\mu+1} -t-2 k )^{d-1}2^{-n_{\mu+1}-k}\right) \\ \ll & 2^t\,N^{-1}
 \sum_{k=0}^{\lceil \ld N \rceil} k^{(d-1)/2} + 2^t\,N^{-1} r \sum_{\ell=0}^{\lceil \ld N \rceil} (2\ell - t)^{d-1}2^{-\ell}  \\ \ll & 2^t \,N^{-1} (\log N)^{(d+1)/2}.
\end{align*}

It remains to estimate the terms where $k \ge \ld N$. As in \cite{BM15} we treat the volume part of the discrepancy function separately from the counting part.

For the volume part $L(\bsx) = x_1 \cdots x_d$ we have $|\langle L, h_{\bsj, \bsm} \rangle | \asymp 2^{-2 |\bsj|}$. Hence
\begin{align*}
& \left\| \sum_{\substack{k=0 \\ k \ge \ld N }}^\infty \sum_{\substack{ \bsj \in \mathbb{N}_{-1}^d \\ |\bsj| = k } } 2^{|\bsj|} \sum_{\bsm \in  \D_{\bsj}} \langle L, h_{\bsj, \bsm} \rangle h_{\bsj,\bsm}  \Bigg| \exp(L^{2/(d-1)}) \right\| \\ \le & \sum_{\substack{ k = 0 \\ k \ge \ld N } }^\infty \left\| \sum_{\substack{ \bsj \in \mathbb{N}_{-1}^d \\ |\bsj| = k}} 2^{|\bsj|} \sum_{\bsm \in  \D_{\bsj}} \langle L, h_{\bsj, \bsm} \rangle h_{\bsj,\bsm}  \Bigg| \exp(L^{2/(d-1)}) \right\| \\ \le & \sum_{ \substack{ k=0 \\ k \ge \ld N } }^\infty  \left\|  \left( \sum_{\substack{ \bsj \in \mathbb{N}_{-1}^d \\ |\bsj | = k }} 2^{2|\bsj|}  \sum_{\bsm \in \D_{\bsj}} \left| \langle L, h_{\bsj, \bsm} \rangle \right|^2 \chi_{I_{\bsj, \bsm}} \right)^{1/2}  \Bigg| L_\infty([0,1)^d) \right\| \\ \ll & \sum_{ \substack{ k=0 \\ k \ge \ld N } }^\infty   \left( \sum_{\substack{ \bsj \in \mathbb{N}_{-1}^d \\ |\bsj | = k }}  2^{-2|\bsj|}  \right)^{1/2}  \\ \ll & \sum_{\substack{ k = 0 \\ k \ge \ld N}}^\infty k^{(d-1)/2}2^{-k} \\ \ll & 
N^{-1} (\log N)^{(d-1)/2}.
\end{align*}

It remains to estimate the counting part $C(\bsx) =  N^{-1}  \sum_{\bsz \in \P_{N,d}} \chi_{[\bszero, \bsx)}(\bsz)$. Let $n \in \mathbb{N}$ be such that $2^{n-1} < N \le 2^n$. Let again $\mathcal{J}$ be the set of all dyadic intervals $I_{\bsj, \bsm}$ with $|\bsj| \ge n$, i.e. $|I_{\bsj, \bsm}| \le 2^{-n}$, such that $\langle C, h_{\bsj, \bsm} \rangle \neq 0$. As above, this implies that for $I_{\bsj, \bsm} \in \mathcal{J}$ we have that at least one point of $\P_{N,d}$ lies in the interior of $I_{\bsj, \bsm}$.  Hence $j_k \le 2n$, or in other words $|I_{j_k, m_k}| \ge 2^{-2n}$, for each coordinate $k = 1, 2, \ldots, d$.

We also define the unique parent of each $I_{\bsj, \bsm} \in \mathcal{J}$, denoted by $\widetilde{I}_{\bsj, \bsm}$, which satisfies: (i) $I_{\bsj, \bsm} \subseteq \widetilde{I}_{\bsj', \bsm'}$; (ii) $|\bsj'| = n$, i.e. $|\widetilde{I}_{\bsj', \bsm'}| = 2^{-n}$; and (iii) $j_k = j'_k$ (which implies that $I_{j_k, m_k} = \widetilde{I}_{j'_k, m'_k}$) for all $k = 1, 2, \ldots, d$. In other words, to find the parent, we expand the $d$-th side of $I_{\bsj, \bsm}$ to a dyadic interval such that the volume of the resulting interval has volume $2^{-n}$. By reordering the sum over the intervals with respect to their parents we obtain
\begin{equation}\label{eq_sum_C}
\sum_{\substack{\bsj \in \mathbb{N}_{-1}^d \\ |\bsj| \ge n}} 2^{|\bsj|} \sum_{\bsm \in \D_{\bsj}} \langle C, h_{\bsj, \bsm} \rangle h_{\bsj, \bsm} = \sum_{\substack{ \widetilde{I}_{\bsj', \bsm'}: |\bsj'| = n \\ j'_k \le 2n: k= 1, \ldots, d}} \sum_{\substack{ I_{\bsj, \bsm} \subseteq \widetilde{I}_{\bsj', \bsm'} \\ j_k = j'_k: k = 1, \ldots, d-1}} 2^{|\bsj|} \langle C, h_{\bsj, \bsm} \rangle h_{\bsj, \bsm}.
\end{equation}

For an arbitrary, but fixed, parent interval $\widetilde{I}_{\bsj', \bsm'}$ we consider the innermost sum above
\begin{equation}\label{eq_p}
\sum_{ \substack{ I_{\bsj, \bsm} \subseteq \widetilde{I}_{\bsj', \bsm'} \\ j_k = j'_k: k = 1, \ldots, d-1 }} 2^{|\bsj|} \langle C, h_{\bsj, \bsm} \rangle h_{\bsj, \bsm} =  N^{-1}  \sum_{\bsz \in \P_{N,d} \cap \widetilde{I}_{\bsj', \bsm'}} \sum_{\substack{ I_{\bsj,\bsm} \subseteq \widetilde{I}_{\bsj', \bsm'} \\ j_k = j'_k: k = 1, \ldots, d-1}} 2^{|\bsj|} \langle \chi_{[\bsz, \bsone)}, h_{\bsj, \bsm} \rangle h_{\bsj, \bsm}.
\end{equation}
The summands in the last sum split into products of one-dimensional factors
\begin{align*}
2^{|\bsj|} \langle \chi_{[\bsz, \bsone)}, h_{\bsj, \bsm} \rangle h_{\bsj, \bsm}(\bsx) &= \prod_{k=1}^d 2^{j_k} \langle \chi_{[z_k, 1)}, h_{j_k, m_k} \rangle h_{j_k, m_k}(x_k) \\ &= \left( \prod_{k=1}^{d-1} 2^{j'_k} \langle \chi_{[z_k,1)}, h_{j'_k, m'_k} \rangle h_{j'_k, m'_k}(x_k) \right) 2^{j_d} \langle \chi_{[z_d, 1)}, h_{j_d, m_d} \rangle h_{j_d, m_d}(x_d) \\ &= 2^{|\bsj'_\ast|} \langle \chi_{[\bsz_\ast, \bsone)}, h_{\bsj'_\ast, \bsm'_\ast} \rangle 2^{j_d} \langle \chi_{[z_d, 1)}, h_{j_d, m_d} \rangle h_{j_d, m_d}(x_d),
\end{align*}
where by $\ast$ we denote the projection of a $d$-dimensional vector to its first $d-1$ coordinates (i.e., for instance, $\bsz_\ast = (z_1, \ldots, z_{d-1})$ for $\bsz = (z_1, \ldots, z_{d-1}, z_d)$). Thus we can write the innermost sum on the right hand side of \eqref{eq_p} as
\begin{align*}
& \sum_{\substack{ I_{\bsj,\bsm} \subseteq \widetilde{I}_{\bsj', \bsm'} \\ j_k = j'_k: k = 1, \ldots, d-1}} 2^{|\bsj|} \langle \chi_{[\bsz, \bsone)}, h_{\bsj, \bsm} \rangle h_{\bsj, \bsm}(\bsx) \\ &= 2^{|\bsj'_\ast|} \langle \chi_{[\bsz_\ast, \bsone)}, h_{\bsj'_\ast, \bsm'_\ast} \rangle h_{\bsj'_\ast, \bsm'_\ast}(x_\ast) \sum_{I_{j_d, m_d} \subseteq \widetilde{I}_{j'_d, m'_d}} 2^{j_d} \langle \chi_{[z_d, 1)}, h_{j_d, m_d} \rangle h_{j_d, m_d}(x_d).
\end{align*}
The last sum is now the Haar series expansion of the indicator function $\chi_{[z_d, 1)}$ restricted to the interval $\widetilde{I}_{j'_d, m'_d}$ without the constant term, i.e., we have
\begin{equation}\label{eq_haar_i}
\sum_{I_{j_d, m_d} \subseteq \widetilde{I}_{j'_d, m'_d}} 2^{j_d} \langle \chi_{[z_d, 1)}, h_{j_d, m_d} \rangle h_{j_d, m_d}(x_d) = \chi_{\widetilde{I}_{j'_d, m'_d}}(x_d)\ \left(\chi_{[z_d, 1)}(x_d) - 2^{j'_d} \left| [z_d, 1) \cap \widetilde{I}_{j'_d, m'_d} \right| \right).
\end{equation}
Using the triangle inequality it follows that the last expression is bounded pointwise by $2$. Further we have $|2^{|\bsj'_\ast|} \langle \chi_{[\bsz_\ast, \bsone)}, h_{\bsj'_\ast, \bsm'_\ast} \rangle | \le 1$. Since there are at most $2^{\lceil t/2 \rceil}$ points of $\P_{N,d}$ in $\widetilde{I}_{\bsj', \bsm'}$, we obtain that
\begin{equation*}
\sum_{\substack{ I_{\bsj,\bsm} \subseteq \widetilde{I}_{\bsj', \bsm'} \\ j_k = j'_k: k = 1, \ldots, d-1}} 2^{|\bsj|} \langle C, h_{\bsj, \bsm} \rangle h_{\bsj, \bsm}(\bsx) =  N^{-1}  \alpha_{j'_d}(x_d) h_{\bsj'_\ast, \bsm'_\ast}(\bsx_\ast),
\end{equation*}
where $|\alpha_{j'_d}(x_d)| \le 2^{\lceil t/2 \rceil + 1} \ll 1$.

Let now $x_d$ be fixed. Eq. \eqref{eq_haar_i} implies that for a given $(d-1)$-dimensional interval $\widetilde{I}_{\bsj'_\ast, \bsm'_\ast}$ there exists only one $d$-dimensional interval of the form $\widetilde{I}_{\bsj', \bsm'} = \widetilde{I}_{\bsj'_\ast, \bsm'_\ast} \times \widetilde{I}_{j'_d, m'_d}$ with $|\widetilde{I}_{\bsj', \bsm'}| = 2^{-n}$ such that $\alpha_{j'_d}(x_d) \neq 0$. We can now apply \eqref{eq_sum_C} and take the $L_p$-norm in the first $d-1$ variables to obtain
\begin{align*}
& \left\| \sum_{\substack{ \bsj \in \mathbb{N}_{-1}^d \\ |\bsj| \ge n}} 2^{|\bsj|} \sum_{\bsm \in \D_{\bsj}} \langle C, h_{\bsj, \bsm} \rangle h_{\bsj, \bsm} \Bigg| L_p([0,1)^{d-1})  \right\| \\ &=  \left\| \sum_{\substack{ \widetilde{I}_{\bsj', \bsm'} : |\bsj'| = n \\ j'_k \le 2n : k = 1, \ldots, d }} \sum_{\substack{ I_{\bsj, \bsm} \subseteq \widetilde{I}_{\bsj', \bsm'} \\ j_k = j'_k: k = 1, \ldots, d-1  }} 2^{|\bsj|} \langle C, h_{\bsj,\bsm}\rangle h_{\bsj,\bsm}  \Bigg| L_p([0,1)^{d-1}) \right\| \\ &=  N^{-1}  \left\|  \sum_{\substack{ \widetilde{I}_{\bsj'_\ast, \bsm'_\ast}: |\bsj'| = n \\ j'_k \le 2n: k = 1, \ldots, d-1 }} \alpha_{j'_d}(x_d) h_{\bsj'_\ast, \bsm'_\ast}    \Bigg| L_p([0,1)^{d-1} \right\| \\ &\ll  p^{(d-1)/2} \,N^{-1} \left\| \left( \sum_{\substack{ \widetilde{I}_{\bsj'_\ast, \bsm'_\ast}: |\bsj'| = n \\ j'_k \le 2n: k = 1, \ldots, d-1    }}   |\alpha_{j'_d}(x_d)|^2  \chi_{\widetilde{I}_{\bsj'_\ast, \bsm'_\ast}}  \right)^{1/2}    \Bigg| L_p([0,1)^{d-1}) \right\|  \\ &\ll  p^{(d-1)/2} n^{(
d-1)/2}  \,N^{-1} ,
\end{align*}
where in the penultimate step we have employed the $(d-1)$-dimensional Littlewood-Paley inequality (see Lemma~\ref{lem_lp}) and in the last step the fact that the number of choices of $\bsj'_\ast$ is of order $n^{d-1}$ at most.

By integrating this bound with respect to the last variable $x_d$ and using Proposition~\ref{prp_exp_Lp} we obtain
\begin{equation*}
\left\| \sum_{\substack{\bsj \in \mathbb{N}_{-1}^d \\  |\bsj| \ge n}} 2^{|\bsj|} \sum_{\bsm \in \D_{\bsj}} \langle C_{\P_{N,d}}, h_{\bsj, \bsm} \rangle h_{\bsj, \bsm}\Bigg| \exp(L^{2/(d-1)}) \right\| \ll  N^{-1} n^{(d-1)/2} ,
\end{equation*}
from which the result now follows.
\end{proof}

\subsection{The proof of the upper bounds in Theorem~\ref{thm:besov}}\label{sec_besov2}

\begin{proof}[Proof of Eq. \eqref{besov:ubd1} and \eqref{besov:ubd2} in Theorem~\ref{thm:besov}]
According to Proposition~\ref{haarbesovnorm} we have 
\begin{equation}\label{besovdisc}
\|D_{\S_d}^N|S_{p,q}^sB\|^q \ll_{p,q,s,d} \sum_{\bsj\in\N_{-1}^d} 2^{|\bsj|(s-1/p+1)q}\left(\sum_{\bsm\in\D_{\bsj}} |\langle D_{\S_d}^N,h_{\bsj,\bsm}\rangle|^p\right)^{q/p}.
\end{equation}
Let $|\bsj|+t/2\ge \ld N$, then according to Lemma~\ref{le131416} we have
\[ |\langle D_{\S_d}^N,h_{\bsj,\bsm}\rangle| \ll 2^{t/2}\,N^{-1} 2^{-|\bsj|} \]
and for a fixed $\bsj$, at least $2^{|\bsj|}-N$ intervals $I_{\bsj,\bsm}$ contain no points of $\P_{N,d}$ (the point set consisting of the first $N$ elements of $\S_d$) and in such cases we have according to Lemma~\ref{le131416}
\[ |\langle D_{\S_d}^N,h_{\bsj,\bsm}\rangle| \ll 2^{-2|\bsj|}. \]
Now we estimate \eqref{besovdisc} by applying Minkowski's inequality to obtain
\begin{eqnarray}\label{besov_small_int}
\lefteqn{\sum_{|\bsj|+t/2 \geq \ld N} 2^{|\bsj|(s-1/p+1)q}\left(\sum_{\bsm\in\D_{\bsj}} |\langle D_{\S_d}^N,h_{\bsj,\bsm}\rangle|^p\right)^{q/p}}\nonumber\\
&\ll & \sum_{|\bsj|+t/2 \geq \ld N} 2^{|\bsj|(s-1/p+1)q} \left(N 2^{pt/2}\,N^{-p} 2^{-p|\bsj|}\right)^{q/p} \nonumber\\
&& +\sum_{|\bsj|+t/2 \geq \ld N} 2^{|\bsj|(s-1/p+1)q} \left(\max\{2^{|\bsj|}-N, 0\} 2^{-2p|\bsj|}\right)^{q/p}\nonumber \\
&\le& N^{q/p-q} 2^{qt/2} \sum_{|\bsj|+t/2 \geq \ld N} 2^{|\bsj|(s-1/p)q} + \sum_{|\bsj|+t/2 \geq \ld N} 2^{|\bsj|(s-1)q}\nonumber \\
&\ll& N^{q/p-q} 2^{qt/2} N^{(s-1/p)q} 2^{-t/2(s-1/p)q} (\log N)^{d-1} + N^{(s-1)q}2^{-t/2(s-1)q}(\log N)^{d-1}\nonumber \\
&\le& 2^{(1-s+1/p)tq/2} N^{(s-1)q}(\log N)^{d-1}.
\end{eqnarray}
Now let $n_r>\ldots>n_{\mu+1}\ge |\bsj|+t/2>n_\mu>\ldots>n_1$. Then according to Lemma~\ref{le131416} we have
\[ |\langle D_{\S_d}^N,h_{\bsj,\bsm}\rangle| \ll 2^t\,N^{-1} \left(2^{-|\bsj|} +  (2n_{\mu+1} - t - 2|\bsj|)^{d-1}2^{-n_{\mu+1}}\right). \]
We use this bound to estimate the terms of \eqref{besovdisc} for which $|\bsj|+t/2 < \ld N$ (again we set $n_0 = 0, \, n_{r+1} = \ld N$). We use Minkowski's inequality
\begin{eqnarray}\label{besov_large_int1} 
\lefteqn{\sum_{|\bsj|+t/2 < \ld N} 2^{|\bsj|(s-1/p+1)q}\left(\sum_{\bsm\in\D_{\bsj}} |\langle D_{\S_d}^N,h_{\bsj,\bsm}\rangle|^p\right)^{q/p}} \nonumber \\
&= & \sum_{\mu=0}^r \sum_{n_{\mu} \le |\bsj|+t/2 < n_{\mu+1}} 2^{|\bsj|(s-1/p+1)q}\left(\sum_{\bsm\in\D_{\bsj}} |\langle D_{\S_d}^N,h_{\bsj,\bsm}\rangle|^p\right)^{q/p}\nonumber \\
&\ll &  \sum_{\mu=0}^r \sum_{n_{\mu} \le |\bsj|+t/2 < n_{\mu+1}} 2^{|\bsj|(s-1/p+1)q} 2^{q|\bsj|/p} \, 2^{qt}\,N^{-q} \left( 2^{-|\bsj|} +  (2n_{\mu+1} - t - 2|j|)^{d-1}2^{-n_{\mu+1}} \right)^{\lceil q\rceil}\nonumber \\
&= & 2^{qt}\,N^{-q}  \sum_{\mu=0}^r \sum_{n_{\mu} \le |\bsj|+t/2 < n_{\mu+1}} 2^{|\bsj|(s+1)q} \sum_{i=0}^{\lceil q\rceil} 2^{-|\bsj|i}2^{-n_{\mu+1}(\lceil q\rceil-i)}(2n_{\mu+1}-t-2|\bsj|)^{(d-1)(\lceil q\rceil-i)}\nonumber \\
&\leq & 2^{qt}\,N^{-q}  \sum_{\mu=0}^r \sum_{i=0}^{\lceil q\rceil}2^{-n_{\mu+1}(\lceil q\rceil-i)} \sum_{n_{\mu} \le |\bsj|+t/2 < n_{\mu+1}} 2^{|\bsj|(sq+\lceil q\rceil-i)}(2n_{\mu+1}-t-2|\bsj|)^{(d-1)(\lceil q\rceil-i)}
\end{eqnarray}
and obtain from Lemma~\ref{index_dim_red} and Lemma~\ref{index_dim_red_log} for $s>0$
\begin{eqnarray}\label{besov_large_int2}
\lefteqn{\sum_{|\bsj|+t/2 < \ld N} 2^{|\bsj|(s-1/p+1)q}\left(\sum_{\bsm\in\D_{\bsj}} |\langle D_{\S_d}^N,h_{\bsj,\bsm}\rangle|^p\right)^{q/p}} \nonumber \\
&\ll & 2^{qt}\,N^{-q}  \sum_{\mu=0}^r \sum_{i=0}^{\lceil q\rceil}2^{-n_{\mu+1}(\lceil q\rceil-i)} 2^{(n_{\mu+1}-t/2)(sq+\lceil q\rceil-i)} (\log N)^{d-1}\nonumber \\
&\le &  2^{qt}\,N^{-q} 2^{-tsq/2-t\lceil q\rceil/2}  \sum_{\mu=0}^r 2^{n_{\mu+1}sq}\sum_{i=0}^{\lceil q\rceil}2^{ti/2}(\log N)^{d-1}\nonumber \\
&\ll &  2^{qt-tsq/2}\,N^{-q}  2^{n_{r+1}sq} (\log N)^{d-1}\nonumber \\
&= & 2^{tq(1-s/2)} \, N^{(s-1)q}(\log N)^{d-1}.
\end{eqnarray}
The case $s=0$ needs to be handled slightly differently. We continue from \eqref{besov_large_int1} and obtain
\begin{eqnarray}\label{besov_large_int_s0} 
\lefteqn{\sum_{|\bsj|+t/2 < \ld N} 2^{|\bsj|(-1/p+1)q}\left(\sum_{\bsm\in\D_{\bsj}} |\langle D_{\S_d}^N,h_{\bsj,\bsm}\rangle|^p\right)^{q/p}} \nonumber \\
&\ll & 2^{qt}\,N^{-q}  \sum_{\mu=0}^r \sum_{i=0}^{\lceil q\rceil-1}2^{-n_{\mu+1}(\lceil q\rceil-i)} \sum_{n_{\mu} \le |\bsj|+t/2 < n_{\mu+1}} 2^{|\bsj|(\lceil q\rceil-i)}(2n_{\mu+1}-t-2|\bsj|)^{(d-1)(\lceil q\rceil-i)}\nonumber \\
&&\quad + 2^{qt}\,N^{-q} \sum_{\mu=0}^r \sum_{n_{\mu} \le |\bsj|+t/2 < n_{\mu+1}} (2n_{\mu+1}-t-2|\bsj|)\nonumber \\
&\ll & 2^{qt} \, N^{-q}(\log N)^d + 2^{qt}N^{-q} \sum_{\mu=0}^r (n_{\mu+1} - n_{\mu}) (\log N)^{d-1}\nonumber \\
&\ll & 2^{qt} \, N^{-q}(\log N)^d.
\end{eqnarray}
Combining \eqref{besov_small_int}, \eqref{besov_large_int2} and \eqref{besov_large_int_s0} we obtain
\begin{align*}
 \|D_{\S_d}^N | S_{p,q}^sB\|^q \ll_{p,q,s,d} \begin{cases} 2^{tq} N^{-q}\,\left(\log N\right)^{d} & \mbox{ if } s=0, \\ 2^{tq} N^{(s-1)q}\,\left(\log N\right)^{d-1} & \mbox{ if } 0<s<1/p, \end{cases}
\end{align*}
and the result follows.
\end{proof}

\subsection{The proof of the upper bounds in Theorem~\ref{thm:triebel}}

\begin{proof}[Proof of Eq. \eqref{triebel:ubd1} and \eqref{triebel:ubd2} in Theorem~\ref{thm:triebel}]
From Proposition~\ref{prp_emb_BF} we know that \[S_{\max(p,q),q}^s B\hookrightarrow S_{p,q}^s F. \] Thus we have
\[ \|D_{\S_d}^N | S_{p,q}^sF\| \ll \|D_{\S_d}^N | S_{\max(p,q),q}^sB\| \ll \begin{cases} 2^{t} N^{-1}\,\left(\log N\right)^{d/q} & \mbox{ if } s=0, \\ 2^{t} N^{(s-1)}\,\left(\log N\right)^{(d-1)/q} & \mbox{ if } 0<s<1/\max(p,q). \end{cases} \]
\end{proof}

\section{Lower bounds}

In this section we prove the lower bounds on the respective norms of the discrepancy function of arbitrary sequences in the unit cube.

\subsection{The proof of the lower bound in Theorem~\ref{thm:bmo}}\label{sec_bmo2}

To prove Eq.~\eqref{bmo:lowbd} we actually prove a lower bound for a norm close to the $L_2$-norm of $f$ which is dominated by the $\bmo$-norm.
For simplicity of notation, let us consider here the $L_2$-normalized version of the Haar system consisting of the functions 
$g_{\bsj,\bsm} = 2^{\frac{|\bsj|}{2}} h_{\bsj,\bsm}$. Moreover, for a subset ${\cal E}$ of indices $(\bsj,\bsm)$ with $\bsj\in \N_{-1}^d,\bsm\in\D_{\bsj}$
let $P_{\cal E}f$ denote the orthogonal projection of $f \in L_2 ([0,1]^d)$ onto the span of $\{ g_{\bsj,\bsm} \,:\, (\bsj,\bsm) \in {\cal E} \}$
and let $ \| f | {\cal E}\| = \| P_{\cal E}f | L_2 \|$. Then Parseval's equality yields
\[ \| f | {\cal E}\|^2 = \sum_{(\bsj,\bsm) \in {\cal E} } \langle f, g_{\bsj,\bsm} \rangle^2.\]
We also fix the index sets
\[ {\cal D}^d = \big\{ (\bsj,\bsm) \,:\, \bsj\in \N_{-1}^d,\bsm\in\D_{\bsj}\big\} \quad \text{and}\quad
   {\cal D}_0^d = \big\{ (\bsj,\bsm) \,:\, \bsj\in \N_{0}^d,\bsm\in\D_{\bsj}\big\}.\]

Considering $U=[0,1)^d$ in the definition \eqref{eq:defbmo} of the $\bmo$-norm, we obviously have
\[ \| f | \bmo^d \| \ge \| f | {\cal D}_0^d\|.\]
Now the lower bound \eqref{bmo:lowbd} in Theorem~\ref{thm:bmo} follows immediately from the next lemma.

\begin{lem}\label{lbdDpts}
For every $d \in \mathbb{N}$ there exists a positive constant $c_d>0$ with the following property: for every infinite sequence $\S_d$ in $[0,1)^d$ we have \[\|D_{\S_d}^N| {\cal D}_0^d\| \ge c_d  \,N^{-1} (\log N)^{d/2}\ \ \ \mbox{ for infinitely many $N \in \mathbb{N}$.}\]
\end{lem}

For the proof we combine the following lemma with the method of Proinov.

\begin{lem}\label{le0}
 Let $f:[0,1]^{d+1} \rightarrow \mathbb{R}$ be in $L_2([0,1]^d)$. 
 Then we have \[\|f|{\cal D}_0^{d+1}\| \le \sup_{t_{d+1} \in [0,1]} \|f(\cdot,t_{d+1})|{\cal D}_0^d\|.\] 
\end{lem}

\begin{proof}
 We are going to show the identity
 \begin{equation}\label{eq:idenorm}
  \int_0^1 \|f(\cdot,t_{d+1})|{\cal D}_0^d\|^2 \dint t_{d+1} = \| f | {\cal E}^{d+1}\|^2,
 \end{equation}
 where 
 \[ {\cal E}^{d+1} = \big\{ (\bsj,\bsm) \,:\, \bsj\in \N_{0}^d \times \N_{-1},\bsm\in\D_{\bsj}\big\}.\]
 Then the claimed inequality is a consequence of \eqref{eq:idenorm} together with the obvious inequalities
 \[ \|f|{\cal D}_0^{d+1}\| \le \| f | {\cal E}^{d+1}\| 
  \quad \text{and} \quad 
  \int_0^1 \|f(\cdot,t_{d+1})|{\cal D}_0^d\|^2 \dint t_{d+1} \le \sup_{t_{d+1} \in [0,1]} \|f(\cdot,t_{d+1})|{\cal D}_0^d\|^2. \]
For the proof of \eqref{eq:idenorm}, let 
\[ f = \sum_{(\bsj,\bsm) \in {\cal D}^{d+1}} a_{\bsj,\bsm} g_{\bsj,\bsm} \]
be the Haar decomposition of $f$ which converges in $L_2$ and, therefore, almost everywhere.
For each $(\bsj,\bsm) \in {\cal D}^{d+1}$ with $\bsj=(j_1,\dots,j_d,j_{d+1})$ and $\bsm=(m_1,\dots,m_d,m_{d+1})$, 
we let $\bsj_*=(j_1,\dots,j_d)$ and $\bsm_*=(m_1,\dots,m_d)$.
Then
\[ f(\cdot,t_{d+1}) = \sum_{(\bsj_*,\bsm_*) \in {\cal D}^{d}} b_{\bsj_*,\bsm_*}(t_{d+1}) g_{\bsj_*,\bsm_*} \]
in $L_2([0,1]^d)$ for almost every $t_{d+1}\in [0,1]$, where
\[ b_{\bsj_*,\bsm_*} (t_{d+1}) = \sum_{(j_{d+1},m_{d+1})\in {\cal D}^1} a_{\bsj,\bsm} g_{j_{d+1},m_{d+1}}(t_{d+1}).\]
It follows that
\[ \|f(\cdot,t_{d+1})|{\cal D}_0^d\|^2 = \sum_{(\bsj_*,\bsm_*) \in {\cal D}_0^{d}} b_{\bsj_*,\bsm_*}(t_{d+1})^2.\]
By integrating over $t_{d+1}\in [0,1]$ and using the orthogonality of the one-dimensional Haar system we finally conclude that
\begin{align*} 
  \int_0^1 \|f(\cdot,t_{d+1})|{\cal D}_0^d\|^2 \dint t_{d+1} 
    &= \sum_{(\bsj_*,\bsm_*) \in {\cal D}_0^{d}} \int_0^1 b_{\bsj_*,\bsm_*}(t_{d+1})^2 \dint t_{d+1}\\
     &= \sum_{(\bsj_*,\bsm_*) \in {\cal D}_0^{d}} \, \sum_{(j_{d+1},m_{d+1)}\in {\cal D}^1} a_{\bsj,\bsm}^2 \\
&= \| f | {\cal E}^{d+1}\|^2.
\end{align*}
\end{proof}

Additionally, we need the following lower bound for finite point sets.

\begin{prp}\label{lbdBMOpts}
For every $d \in \mathbb{N}$ there exists a positive constant $c_d>0$ with the following property: for every integer $N \ge 2$ and every $N$-element point set $\P_{N,d}$ in $[0,1)^d$ we have \[\|D_{\P_{N,d}}|{\cal D}_0^d\| \ge c_d  \,N^{-1} (\log N)^{(d-1)/2}.\]  
\end{prp}

This has been shown in \cite{BLPV09} for $d=2$ and \cite{BM15} for $d \ge 3$ with the $\bmo^d$-norm instead of the ${\cal D}_0^d$-norm. 
But the proofs actually work literally for the ${\cal D}_0^d$-norm. The proof of the following result relies on Proinov's idea.

\begin{lem}\label{le1}
For $d\in \mathbb{N}$ let $\S_d=(\bsy_k)_{k\ge 0}$, where $\bsy_k=(y_{1,k},\ldots,y_{d,k})$ for $k \in \mathbb{N}_0$, be an arbitrary sequence in $[0,1)^d$. 
Then for every $N\in \mathbb{N}$ there exists an $n \in \{1,2,\ldots,N\}$ such that 
\[n \|D_{\S_d}^n|{\cal D}_0^d\| +1 \ge N \|D_{\P_{N,d+1}}|{\cal D}_0^{d+1}\|\] 
where $\P_{N,d+1}$ is the finite point set in $[0,1)^{d+1}$ consisting of the points 
\[\bsx_k=(y_{1,k},\ldots,y_{d,k},k/N) \ \ \mbox{ for }\ k=0,1,\ldots ,N-1.\]
\end{lem}

\begin{proof}
Choose $n \in \{1,2,\ldots,N\}$ such that 
\begin{equation}\label{choicn}
\|n D_{\S_d}^n|{\cal D}_0^d\| =\max_{k=1,2,\ldots,N} \|kD_{\S_d}^k|{\cal D}_0^d\|.
\end{equation}
Consider a sub-interval of the $(d+1)$-dimensional unit cube of the form $E=\prod_{i=1}^{d+1}[0,t_i)$ with $\bst=(t_1,t_2,\ldots,t_{d+1}) \in [0,1)^{d+1}$ and put \[m=m(t_{d+1}):=\lceil N t_{d+1}\rceil \in \{1,2,\ldots,N\}.\] Then a point $\bsx_k$, $k=0,1,\ldots, N-1$, belongs to $E$, if and only if $\bsy_k \in \prod_{i=1}^d[0,t_i)$ and $0 \le k < N t_{d+1}$. Denoting $E_*=\prod_{i=1}^d[0,t_i)$ we have \[\sum_{\bsz \in \P_{N,d+1}} \chi_E(\bsz)=\sum_{\bsy \in \P_{m,d}} \chi_{E_*}(\bsy),\] where $\P_{m,d}$ is the point set which consists of the first $m$ terms of the sequence $\S_d$. Therefore we have
\begin{align*}
N D_{\P_{N,d+1}}(\bst) &=  \sum_{\bsz \in \P_{N,d+1}} \chi_E(\bsz) -N  t_1 t_2\cdots t_{d+1}\\
&=  \sum_{\bsy \in \P_{m,d}} \chi_{E_*}(\bsy) - m t_1 t_2 \cdots t_d + m t_1 t_2 \cdots t_d-N  t_1 t_2\cdots t_{d+1}\\
&=  m D_{\P_{m,d}}(\bst_*)+ t_1 t_2 \cdots t_d (m-N t_{d+1}),
\end{align*}
where $\bst_*=(t_1,\ldots,t_d)$.
With 
\[ \varphi(\bst) = t_1 t_2 \cdots t_d (m-N t_{d+1}) = t_1 t_2 \cdots t_d (m(t_{d+1})-N t_{d+1})\]
we obtain
\begin{align}\label{formle1}
N D_{\P_{N,d+1}}(\bst) = m D_{\P_{m,d}}(\bst_*) + \varphi(\bst).
\end{align} 
From the definition of $m$ it is clear that $|m-N t_{d+1}| \le 1$ and, therefore, also $| \varphi(\bst) | \le 1$.
It follows that
\begin{align*} 
    N \|D_{\P_{N,d+1}}|{\cal D}_0^{d+1}\| &\le   \|m D_{\S_d}^m |{\cal D}_0^{d+1}\| + \| \varphi | {\cal D}_0^{d+1} \| \\
        &\le   \|m D_{\S_d}^m |{\cal D}_0^{d+1}\| + \| \varphi | L_2 \| \\
&\le   \|m D_{\S_d}^m |{\cal D}_0^{d+1}\| + 1.
\end{align*}
Using Lemma~\ref{le0} we obtain \[\|m(\cdot) D_{\S_d}^{m(\cdot)}|{\cal D}_0^{d+1}\|+1 \le \sup_{t_{d+1} \in [0,1]} \|m(t_{d+1}) D_{\S_d}^{m(t_{d+1})}|{\cal D}_0^d\|+1 \le \|n D_{\S_d}^n | {\cal D}_0^d\|+1.\]
This implies the result.
\end{proof}

Now we can give the proof of Lemma~\ref{lbdDpts} which also finishes the proof of Eq.~\eqref{bmo:lowbd} in Theorem~\ref{thm:bmo}.

\begin{proof}[Proof of Lemma~\ref{lbdDpts}]
We use the notation from Lemma~\ref{le1}. For the finite point set $\P_{N,d+1}$ in $[0,1)^{d+1}$ we obtain from Proposition~\ref{lbdBMOpts} that 
\[N \|D_{\P_{N,d+1}}|{\cal D}_0^{d+1}\| \ge c_{d+1} (\log N)^{\frac{d}{2}}\] 
for some real $c_{d+1}>0$ which is independent of $N$. According to Lemma~\ref{le1} there exists an $n \in \{1,2,\ldots,N\}$ such that
\begin{equation}\label{eq1}
n \| D_{\S_d}^n | {\cal D}_0^d\| + 1 \ge N \|D_{\P_{N,d+1}}|{\cal D}_0^{d+1}\| \ge c_{d+1} (\log N)^{\frac{d}{2}}.
\end{equation}
Thus we have shown that for every  $N \in \mathbb{N}$ there exists an $n \in \{1,2,\ldots,N\}$ such that
\begin{equation}\label{eq2}
n \| D_{\S_d}^n | {\cal D}_0^d \| + 1 \ge c_{d+1} (\log n)^{\frac{d}{2}}.
\end{equation}
It remains to show that \eqref{eq2} holds for infinitely many $n \in \mathbb{N}$. 
Assume to the contrary that \eqref{eq2} holds for finitely many $n \in \mathbb{N}$ only and let $m$ be the largest integer with this property. 
Then choose $N \in \mathbb{N}$ large enough such that 
\[c_{d+1} (\log N)^{\frac{d}{2}} > \max_{k=1,2,\ldots,m} k \|D_{\S_d}^k|{\cal D}_0^d\|+1.\] 
For this $N$ we can find an $n \in \{1,2,\ldots,N\}$ for which \eqref{eq1} and \eqref{eq2} hold true. 
However, \eqref{eq1} implies that $n > m$ which leads to a contradiction since $m$ is the largest integer such that \eqref{eq2} is true. 
Thus we have shown that \eqref{eq2} holds for infinitely many $n \in \mathbb{N}$ and this completes the proof.
\end{proof}

\subsection{The proof of the lower bounds in Theorems \ref{thm:besov} and \ref{thm:triebel}}

Before we give the proof we collect some auxiliary results.

\begin{lem} \label{lemsge0}
 Let $p,q\in(0,\infty)$ and $0\le s<1/p$. For every $d\in\N$ there exists a positive constant $c_{p,q,s,d}>0$ with the following property: for every infinite sequence $\S_d$ in $[0,1)^d$ and every $N\in\N$ we have 
\[ \|D_{\S_d}^N|S_{p,q}^s B([0,1)^d)\| \ge c_{p,q,s,d} \, N^{s-1} (\log N)^{(d-1)/q}. \]
\end{lem}

\begin{proof}
 We apply Proposition~\ref{haarbesovnorm} and the now classical idea of Roth's proof of the bound on the $L_2$-discrepancy \cite{R54}. Fix an arbitrary infinite sequence $\S_d$. Then for every $N\in\N$ we can choose $n\in\N$ such that $2N\le 2^n < 4N$. Let $\bsj\in\N_0^d$ with $|\bsj|=n$. Then there are at least $2^{n-1}=2^{|\bsj|-1}$ values of $\bsm\in\D_{\bsj}$ for which $I_{\bsj,\bsm}\cap \P_{N,d} = \emptyset$, where $\P_{N,d}$ is again the set consisting of the first $N$ elements of $\S_d$. For such $I_{\bsj,\bsm}$ which do not contain any points of $\P_{N,d}$ we give the Haar coefficients as $|\langle D_{\S_d}^N,h_{\bsj,\bsm}\rangle| \asymp 2^{-2|\bsj|}$, relying on Lemma~\ref{lem_haar_coeff_vol} and \ref{lem_haar_coeff_number}. We apply Proposition~\ref{haarbesovnorm} to obtain (with some constant $c>0$)
\begin{align*}
 \|D_{\S_d}^N|S_{p,q}^s B([0,1)^d)\|^q & \asymp \sum_{\bsj\in\N_{-1}^d} 2^{|\bsj|(s-1/p+1)q}\left(\sum_{\bsm\in\D_{\bsj}} |\langle D_{\S_d}^N,h_{\bsj,\bsm}\rangle|^p\right)^{q/p} \\
 & \ge c\,\sum_{|\bsj|=n} 2^{|\bsj|(s-1/p+1)q} \, 2^{-2|\bsj|q} \, 2^{(|\bsj|-1)q/p} \\
 & \asymp \sum_{|\bsj|=n} 2^{|\bsj|(s-1)q} \asymp n^{d-1} 2^{n(s-1)q}.
\end{align*}
\end{proof}

\begin{lem} \label{besovsup}
 Let $p,q\in(0,\infty)$. Let $f: [0,1]^{d+1} \rightarrow \R$ be in $S_{p,q}^0 B([0,1]^d)$. Then we have
\[ \|f|S_{p,q}^0 B([0,1]^{d+1}) \| \le \sup_{t_{d+1}\in[0,1]} \|f(\cdot,t_{d+1})|S_{p,q}^0 B([0,1]^d)\|. \]
\end{lem}

\begin{proof}
 We write $\langle f,g \rangle_d$ to keep track of the dimension. We have
\[ |\langle f,h_{\bsj,\bsm}\rangle_{d+1}| = |\langle\langle f,h_{\bsj',\bsm'}\rangle_{d},h_{\bsj_{d+1},\bsm_{d+1}}\rangle_1| \le \sup_{t_{d+1}\in[0,1]}|\langle f(\cdot,t_{d+1}),h_{\bsj',\bsm'}\rangle_d| 2^{-\max(j_{d+1},0)}. \]
The result then follows from the fact that $\D_{j_{d+1}}$ contains $2^{j_{d+1}}$ elements and Proposition~\ref{haarbesovnorm}.
\end{proof}

We need the following lower bound for finite point sets from \cite{T10}.

\begin{prp} \label{besovfinitetriebel}
 For every $d\in\N$ and all $1\le p,q <\infty$ and $1/p-1<s<1/p$ there exists a positive constant $c_{p,q,s,d}>0$ with the following property: for every integer $N\ge 2$ and every $N$-element point set $\P_{N,d}$ in $[0,1)^d$ we have
 \[ \|D_{\P_{N,d}}|S_{p,q}^s B([0,1)^d)\| \ge c_{p,q,s,d}\, N^{s-1}\, (\log N)^{(d-1)/q}. \]
\end{prp}

\begin{lem} \label{proinovbesov}
 Let $1\le p,q<\infty$. For $d \in\N$ let $\S_d=(\bsy_k)_{k\ge 0}$, where $\bsy_k=(y_{1,k},\ldots,y_{d,k})$ for $k\in \N_0$, be an arbitrary sequence in $[0,1)^d$. Then for every $N\in\N$ there exists an $n\in\{1,2,\ldots,N\}$ such that
\[ n\| D_{\S_d}^n|S_{p,q}^0 B([0,1)^d)\| +1 \ge N \|D_{\P_{N,d+1}}|S_{p,q}^0 B([0,1)^{d+1})\|, \]
where $\P_{N,d+1}$ is the finite point set in $[0,1)^{d+1}$ consisting of the points
\[ \bsx_k = (y_{1,k},\ldots,y_{d,k},k/N) \text{ for } k=0,1,\ldots,N-1. \]
\end{lem}

\begin{proof}
 The proof is analogous to the proof of Lemma~\ref{le1}. We choose $n\in\{1,2,\ldots,N\}$ such that
\[ \|n D_{\S_d}^n |S_{p,q}^0 B([0,1)^d)\| = \max_{k=1,\ldots,N} \|k D_{\S_d}^k|S_{p,q}^0 B([0,1)^d)\|. \]
For $\bst\in[0,1)^{d+1}$ we have \eqref{formle1}
\[ N D_{\P_{N,d+1}}(\bst) = m D_{\P_{m,d}}(\bst_*) +\varphi(\bst) \]
(see the proof of Lemma~\ref{le1} for the notation). Then analogously to the proof of Lemma~\ref{le1}  we show that
\[ N\|D_{\P_{N,d+1}}|S_{p,q}^0 B([0,1)^{d+1})\| \le \|m D_{\S_d}^m|S_{p,q}^0 B([0,1)^{d+1})\|+1. \]
Finally we use Lemma~\ref{besovsup} to obtain
\begin{align*} \|m(\cdot) D_{\S_d}^{m(\cdot)}|S_{p,q}^0 B([0,1)^{d+1})\| &\le \sup_{t_{d+1}\in [0,1]} \|m(t_{d+1}) D_{\S_d}^{m(t_{d+1})}|S_{p,q}^0 B([0,1)^d)\| \\&\le \|n D_{\S_d}^n|S_{p,q}^0 B([0,1)^d)\|. \end{align*}
\end{proof}

\begin{proof}[Proof of Eq.~\eqref{besov:lbd1} and \eqref{besov:lbd2} in Theorem~\ref{thm:besov}]
The case where $s>0$ is covered by Lemma~\ref{lemsge0}, hence we only have to consider the case $s=0$. We use the notation from Lemma~\ref{le1}. For the finite point set $\P_{N,d+1}$ in $[0,1)^{d+1}$  we obtain from Proposition~\ref{besovfinitetriebel} that
\[ N \|D_{\P_{N,d+1}}|S_{p,q}^0 B([0,1)^{d+1})\|\ge c_{p,q,0,d+1} (\log N)^{d/q} \]
for some constant $c_{p,q,0,d+1}>0$ which is independent of $N$. Now we apply Lemma~\ref{proinovbesov} and obtain an $n\in\{1,2,\ldots,N\}$ such that
\begin{align}\label{formbesov}
 n \|D_{\S_d}^n|S_{p,q}^0 B([0,1)^d)\|+1\ge N\|D_{\P_{N,d+1}}|S_{p,q}^0 B([0,1)^{d+1})\|\ge c_{p,q,0,d+1} (\log N)^{d/q}.
\end{align}
The proof that there are infinitely many $n\in\N$ for which the \eqref{formbesov} holds works analogously to the proof of Lemma~\ref{lbdDpts}.
\end{proof}

\begin{proof}[Proof of Eq.~\eqref{triebel:lbd1} and \eqref{triebel:lbd2} in Theorem~\ref{thm:triebel}]
 We apply Proposition~\ref{prp_emb_BF}, namely $S_{p,q}^s F\hookrightarrow S_{\min(p,q),q}^s B$ and we have for a constant $c>0$ from Theorem~\ref{thm:besov}
\begin{align*} \|D_{\S_d}^N | S_{p,q}^sF\| &\ge c\, \|D_{\S_d}^N | S_{\min(p,q),q}^sB\| \\&\ge c_{p,q,s,d}\, \begin{cases} N^{-1}\,\left(\log N\right)^{d/q} & \mbox{ if } s=0, \\ N^{(s-1)}\,\left(\log N\right)^{(d-1)/q} & \mbox{ if } 0<s<1/\min(p,q). \end{cases} \end{align*}
\end{proof}

\subsection*{Acknowledgement}
The authors would like to thank William Chen for his very helpful remarks and suggestions.


\noindent {\sc Josef Dick} 

\noindent School of Mathematics and Statistics, The University of New South Wales, Sydney NSW 2052, Australia, 
email: josef.dick(AT)unsw.edu.au\\

\noindent {\sc Aicke Hinrichs}

\noindent Institut f\"ur Funktionalanalysis, Johannes Kepler Universit\"at Linz, Altenbergerstra{\ss}e 69, 4040 Linz, \"Osterreich, 
email: aicke.hinrichs(AT)jku.at\\

\noindent {\sc Lev Markhasin}

\noindent Institut f\"ur Stochastik und Anwendungen, Universit\"at Stuttgart, Pfaffenwaldring 57, 70569 Stuttgart, Deutschland, 
email: lev.markhasin(AT)mathematik.uni-stuttgart.de\\

\noindent {\sc Friedrich Pillichshammer}

\noindent Institut f\"ur Finanzmathematik und angewandte Zahlentheorie, Johannes Kepler Universit\"at Linz, Altenbergerstra{\ss}e 69, 4040 Linz, \"Osterreich,
email: friedrich.pillichshammer(AT)jku.at

\end{document}